\documentclass{amsart}
\usepackage{amscd, amssymb, amsmath, amsthm, graphics}
\usepackage{graphicx,psfrag}

\begin{document}

\newtheorem{theorem}{Theorem}[section]
\newtheorem{lemma}[theorem]{Lemma}
\newtheorem{proposition}[theorem]{Proposition}
\newtheorem{corollary}[theorem]{Corollary}
\newtheorem*{hyperbolic_thm}{Hyperbolic Theorem (\ref{minimal_element_orbifold})}
\newtheorem*{toroidal_thm}{Toroidal Examples (\ref{bounded}, \ref{closed})}
\theoremstyle{definition}
\newtheorem{definition}[theorem]{Definition}
\newtheorem{notation}[theorem]{Notation}
\newtheorem{exercise}[theorem]{Exercise}
\newtheorem{question}[theorem]{Question}
\theoremstyle{remark}
\newtheorem{remark}[theorem]{Remark}
\newtheorem{example}[theorem]{Example}
\numberwithin{equation}{section}

\def\Q{{\mathbb Q}}
\def\Z{{\mathbb Z}}
\def\R{{\mathbb R}}
\def\RP{{\mathbb{RP}}}
\def\H{{\mathbb H}}
\def\F{{\mathcal F}}
\def\G{{\mathcal G}}
\def\MCG{{\text{MCG}}}
\def\rank{{\text{rank}}}
\def\isom{{\text{Isom}}}
\def\vol{{\text{vol}}}
\def\id{{\text{id}}}
\def\Sol{{\text{Sol}}}
\def\homeo{{\text{Homeo}}}
\def\PSL{{\text{PSL}}}
\def\GL{{\text{GL}}}
\def\C{{\mathbb C}}
\def\flip{{\sim}} 

\newcommand{\norm}[1]{\left\Vert#1\right\Vert}
\newcommand{\abs}[1]{\left\vert#1\right\vert}
\newcommand{\set}[1]{\left\{#1\right\}}
\newcommand{\Real}{\mathbb R}
\newcommand{\eps}{\varepsilon}
\newcommand{\To}{\longrightarrow}
\newcommand{\BX}{\mathbf{B}(X)}
\newcommand{\A}{\mathcal{A}}

\title{On fibered commensurability}
\author{Danny Calegari}
\address{Department of Mathematics \\ Caltech \\ Pasadena CA 91125 USA}
\email{dannyc@its.caltech.edu}
\author{Hongbin Sun}
\address{Department of Mathematics \\ Princeton University \\ Princeton NJ 08544 USA}
\email{hongbins@math.princeton.edu}
\author{Shicheng Wang}
\address{School of Mathematical Sciences \\ Peking
University \\ Beijing 100871, China}
\email{wangsc@math.pku.edu.cn}

\date{\today}
\begin{abstract}
This paper initiates a systematic study of the relation of
commensurability of surface automorphisms, or equivalently, fibered
commensurability of $3$-manifolds fibering over $S^1$. We show that
every hyperbolic fibered commensurability class contains a unique
minimal element. The situation for toroidal
manifolds is more complicated, and we illustrate a range of
phenomena that can occur in this context.
\end{abstract}
\maketitle

\section{Introduction}

The main purpose of this paper is to study the equivalence relation of {\em commensurability} of
surface automorphisms. Informally, two surface automorphisms are commensurable if they lift
to automorphisms of a finite covering surface that have nontrivial common powers. Equivalently, a
surface automorphism determines a foliation of a $3$-manifold by closed surfaces, and two automorphisms
are commensurable if their corresponding $3$-manifolds admit common finite covers for which the
pulled-back foliations are isotopic. Thus commensurability of surface automorphisms is a special
case of the study of commensurability of $3$-manifolds equipped with a certain kind of geometric
structure; again informally, we call this commensurability relation {\em fibered commensurability}.

The relation of commensurability of $3$-manifolds is well-studied,
see e.g.\ \cite[Chap. 6]{Thurston_notes}, \cite{Borel, Macbeath,
Neumann, Behrstock_Neumann} and so on. When studying
commensurability in a given context, the most important distinction
to make is between those commensurability classes that admit {\em
finitely many} minimal elements, and those that admit {\em
infinitely} many. For example, amongst hyperbolic $3$-manifolds,
this is precisely the distinction between nonarithmetic and
arithmetic commensurability classes, see e.g.\  \cite{Margulis,
Borel}. This distinction has a cleaner statement if one is prepared
to work in the category of orbifolds: each commensurability class of
nonarithmetic hyperbolic $3$-manifolds contains a unique minimal
element.

Fibered commensurability is more rigid than (ordinary) commensurability. However, a given $3$-manifold
can fiber in infinitely many different ways. For Seifert manifolds, there is exactly one
fibered commensurability class of surface bundles of all closed (resp.\ with torus boundary) Seifert
fibered manifolds whose fiber has negative Euler characteristic, and this class contains
infinitely many minimal elements. On the other hand, in the hyperbolic world
we obtain the following theorem:

\begin{hyperbolic_thm}
Every commensurability class of hyperbolic fibered pairs contains a unique (orbifold) minimal element.
\end{hyperbolic_thm}

An immediate corollary is that for a fibered hyperbolic $3$-manifold
$M$, each fibered commensurability class contains at most finitely many
fibrations of $M$; hence $M$ has either one fibered commensurability
class, or infinitely many fibered commensurability classes.

The reducible case is more complicated:

\begin{toroidal_thm}
There are examples of graph manifolds with infinitely many fibered commensurability classes,
and a single graph manifold can fiber in infinitely many ways in a single commensurability class.
\end{toroidal_thm}

As these results suggest, obstructions to commensurability of surface
automorphisms arise from their behavior on pseudo-Anosov orbits, and
near their reducing systems. We describe such obstructions in detail.

In \S~\ref{definition-simple case}, we give basic definitions
and illustrate their meaning, in the special case
of commensurability of spherical and toral automorphisms. We also recall
the Nielsen-Thurston classification of surface automorphisms, and discuss a
``normal form'' for  autormorphisms. This material is standard, and may
be skipped by the expert.

In \S~\ref{hyperbolic_section}, we study fibered commensurability of
hyperbolic manifolds, and prove Theorem~\ref{minimal_element_orbifold}.
We also list some commensurability invariants of
pseudo-Anosov automorphisms (Lemma~\ref{lambda-Delta} and
Proposition~\ref{spectrum_bounded_away_from_zero}), and describe examples that
illustrate their use.

Finally, \S~\ref{reducible_section} and \S~\ref{graph_section} are
devoted to the case of reducible automorphisms, especially of graph
manifolds. In \S~\ref{reducible_section} we define certain numerical
commensurability invariants for reducible maps (Theorem
\ref{reducible}, as well as Proposition \ref{Pi-lambda}), and give
many examples. In \S~\ref{graph_section} we give examples of graph
manifolds with infinitely many incommensurable fibrations, including
one with boundary (Example~\ref{bounded}) that also admits
infinitely many {\em commensurable} (but non-isomorphic) fibrations,
and a closed one (Example~\ref{closed}) that admits incommensurable
fibrations of the same genus.

\subsection{Acknowledgments}

The first author was partially supported by NSF grant DMS 0707130.
The third author was supported by grant no. 10631060 of NSF of China
and by Caltech Mathematics Department as a short term scholar.
The content of this paper benefited from conversations with Juan Souto.
We would like to thank the referee for some helpful suggestions.

\section{Fibered commensurability}\label{definition-simple case}

\subsection{Basic definitions}

Let $F$ be a compact surface. An {\em automorphism} $\phi$ of $F$ is an isotopy class of self-homeomorphisms
of $F$. We use the notation $(F,\phi)$ where $\phi$ is an automorphism of $F$.

\begin{remark}
When $F$ has boundary, it is more usual to study isotopy classes of self-homeomorphisms fixed pointwise
on the boundary. However, since we are interested in automorphisms which might permute boundary
components, we adhere to this nonstandard convention.
\end{remark}

One surface automorphism can ``cover'' another in two distinct ways: either topologically (in the
sense that one surface covers the other) or dynamically (in the sense that one automorphism is a
power of another). We consider covering in both senses in the sequel. More formally, we make the
following definition.

\begin{definition}
A pair $(\tilde{F},\tilde{\phi})$ {\em covers} $(F,\phi)$ if there
is a finite cover $\pi:\tilde{F} \to F$ and representative
homeomorphisms $\tilde{f}$ and $f$ of $\tilde{\phi}$ and $\phi$
respectively so that $\pi \circ \tilde{f} = f\circ \pi$ as maps
$\tilde{F} \to F$.
\end{definition}

\begin{remark}
The relation of covering is transitive: if $(F_1,\phi_1)$ covers $(F_2,\phi_2)$, and
$(F_2,\phi_2)$ covers $(F_3,\phi_3)$, then $(F_1,\phi_1)$ covers $(F_3,\phi_3)$. This follows
by appealing to a ``normal form'' for representative homeomorphisms which is compatible with
finite covers. This normal form is well-known, and summarized in Theorem~\ref{NT} and Proposition~\ref{JG}
below.
\end{remark}

An automorphism $\phi$ of $F$ determines an outer automorphism
$\phi_*$ of $\pi_1(F)$ preserving peripheral subgroups, and by the
well-known theorem of Dehn-Nielsen (see \cite{Nielsen}), this
correspondence is a bijection. A cover $\tilde{F}$ determines a
conjugacy class of subgroups $G$ of $\pi_1(F)$, and an automorphism
$\phi$ of $F$ lifts to an automorphism $\tilde{\phi}$ of $\tilde{F}$
if and only if $G$ and $\phi_*(G)$ are conjugate in $\pi_1(F)$.
However, a particular lift $\tilde{\phi}$ depends on a choice of
conjugating element. Thus a finite cover of surfaces $\tilde{F} \to
F$ might determine zero, one, or many covers of automorphisms
$(\tilde{F},\tilde{\phi}) \to (F,\phi)$ (even if $\tilde{\phi}$ is
primitive).

\begin{example}
If $\tilde{F} \to F$ is any finite cover, then $(F,\id)$ is covered
by $(\tilde{F},\psi)$ where $\psi$ is any element of the deck group
of the cover.
\end{example}

\begin{definition}\label{topologically_dynamically_definition}
Two automorphisms $(F_1,\phi_1)$ and $(F_2,\phi_2)$ are {\em
commensurable} if there is a surface $\tilde{F}$, automorphisms
$\tilde{\phi}_1$ and $\tilde{\phi}_2$ of $\tilde{F}$, and nonzero
integers $k_1$ and $k_2$, so that $(\tilde{F},\tilde{\phi}_i)$
covers $(F_i,\phi_i)$ for $i=1,2$, and if $\tilde{\phi}_1^{k_1} =
\tilde{\phi}_2^{k_2}$ as automorphisms of $\tilde{F}$. Moreover say
$(F_1,\phi_1)$ and $(F_2,\phi_2)$  are {\it topologically
commensurable} if $|k_1|=|k_2|=1$, and {\it dynamically
commensurable} if $\tilde F=F_1=F_2$.
\end{definition}

Commensurability of automorphisms is readily seen to be an equivalence relation, and is the main
object of study in this paper.

\medskip

Statements about surfaces and automorphisms can usefully be translated into statements about $3$-manifolds
with certain types of foliations. These objects --- ``fibered pairs'', to be defined below --- admit
natural generalizations to objects called {\em orbifold} fibered pairs, that are awkward to discuss
in the language of surfaces and automorphisms. Certain theorems in this paper are more elegantly stated
and proved in this category. A basic reference for the theory of orbifolds is \cite{Thurston_notes},
Chapter~13.

\begin{definition}
A {\em fibered pair} is a pair $(M,\F)$ where $M$ is a compact $3$-manifold with boundary
a union of tori and Klein bottles, and $\F$ is a foliation by compact surfaces. More generally,
an {\em orbifold fibered pair} is a pair $(O,\G)$ where $O$ is a compact $3$-orbifold, and $\G$
is a foliation of $O$ by compact $2$-orbifolds.
\end{definition}

At interior points (resp.\ boundary points) an orbifold
fibered pair $(O,\G)$ looks locally like the quotient of an open ball in $\R^3$ (resp.\ a relatively
open ball in a vertical half-space) foliated by horizontal planes by a finite group of smooth
foliation-preserving homeomorphisms.

A surface automorphism $(F,\phi)$ determines a fibered pair whose underlying manifold is an $F$
bundle over $S^1$ with monodromy $\phi$, and whose foliation is the foliation by surface fibers
(which are all homeomorphic to $F$). If we want to emphasize its dynamical origin,
we use the notation $[F,\phi]$ in the sequel to denote the fibered pair associated to the
automorphism $(F,\phi)$.

If the underlying orbifold $O$ is {\em good} (i.e.\ it admits a finite manifold cover) then $(O,\G)$
is finitely covered by a pair $(M,\F)$ where $M$ is a manifold, and every leaf of $\F$ is a compact
surface. After passing to a further $2$-fold cover if necessary, we can assume $\F$ is co-orientable,
in which case $M$ fibers over $S^1$ in such a way that the leaves of $\F$ are the fibers.

\begin{definition}\label{fibered_commensurable}
A fibered pair $(\tilde{M},\tilde{\F})$ {\em covers} $(M,\F)$ if there is
a finite covering of manifolds $\pi:\tilde{M} \to M$ such that
$\pi^{-1}(\F)$ is isotopic to $\tilde{\F}$. Two fibered pairs
$(M_1,\F_1)$ and $(M_2,\F_2)$ are {\em commensurable} if there is a
third fibered pair $(\tilde{M},\tilde{\F})$ that covers both.
\end{definition}

If $(M_i,\F_i)$ for $i=1,2$ are fibered pairs with co-orientable foliations, then
they are commensurable in the sense of Definition~\ref{fibered_commensurable}
if and only if the associated surface automorphisms
are commensurable. Thus, the category of fibered pairs enlarges the category of surface
automorphisms in such a way that the definition of commensurability of a surface automorphism
is the same, whichever category we use.

To stress that the definition of commensurability of fibered pairs depends on both the
underlying $3$-manifold and the foliation, we call this equivalence relation {\em fibered commensurability}.

\medskip

The relation of covering is transitive, but it is not yet a partial order because of the existence
of automorphisms of finite order. We must take such examples into account in order
to define minimal elements with respect to commensurability.

\begin{definition}
We say that two fibered pairs $(M, \F)$ and $(N,\G)$ are {\em covering equivalent} if each covers
the other. Call a covering equivalence class {\em minimal} if no representative covers any element
of another covering equivalence class.
\end{definition}

The relation of covering descends to a transitive relation on covering equivalence classes, and
defines a partial order on such classes. Minimal classes are minimal with respect to this
partial order.

\begin{remark} Each covering equivalence class of
fibered pairs $[F, \phi]$ contains exactly one fibered pair unless
$\phi$ is periodic.   In the periodic case, $(F,\phi)$ and
$(G,\psi)$ are in the same covering equivalent class if and only if
$F=G$ and both $\phi$ and $\psi$ generate the same finite cyclic
group. With this understood, in the sequel we are relaxed in our terminology,
and use the word ``minimal element'' when we really mean ``minimal class''.
\end{remark}

\subsection{Simple cases}
For simplicity, we usually restrict attention to the case that $F$ (and therefore $M$) is
closed. However, because of the nature of the theory of surface automorphisms,
to really understand this case we are forced to consider surfaces (and $3$-manifolds) with
boundary, associated to the restrictions of automorphisms to invariant subsurfaces.

Evidently, the sign of $\chi(F)$ is a commensurability invariant of $(F,\phi)$. In the case of
fibered pairs (of good orbifolds), all leaves have the same sign, so we can speak unambiguously
about fibered pairs with spherical, Euclidean, or hyperbolic leaves.
We first discuss the situation when $\chi(F) \ge 0$.

\begin{example}[Spherical automorphisms]
There is one commensurability class consisting of the bundles $S^2 \times S^1$
and $S^2 \tilde{\times} S^1$, each foliated by spheres, and $\RP^3 \# \RP^3$
which can be thought of as an $S^2$ bundle over a mirror orbifold. The
elements $S^2 \tilde{\times} S^1$ and $\RP^3 \# \RP^3$ are minimal.
\end{example}

\begin{example}[Toral automorphisms]
The mapping class group of a torus is isomorphic to $\GL(2,\Z)$, and every automorphism has a
linear representative. An automorphism can be periodic, reducible, or Anosov. From elementary
linear algebra, automorphisms in different classes are not commensurable. We discuss
each case in turn.
\begin{enumerate}
\item{Periodic case: there is only one commensurability class; moreover
there are exactly two minimal elements,
corresponding to the periodic automorphisms of order 4 and 6 on a square and hexagonal
torus respectively.}
\item{Reducible case: as automorphisms, each  map $(T, \phi)$ is
represented by a matrix which can be conjugated into the form
$$ \phi \sim \pm\left( \begin{array}{cc} 1 & n \\ 0 & 1 \\ \end{array} \right)$$
where $n\neq 0$. So there is only one commensurability class and two minimal
elements, corresponding to the conjugacy classes of matrices
$$\phi \sim \left( \begin{array}{cc} 1 &  1\\  0 & 1 \\ \end{array} \right) \text{ or }
\left( \begin{array}{cc} -1 & 1 \\ 0 & -1 \\ \end{array} \right)$$}
\item{Anosov case: the resulting Sol manifolds are commensurable if and
only if they are fibered commensurable, which occurs if and only if
the logarithms of the dilatations of the automorphisms are
commensurable as real numbers. Hence there are infinitely many
fibered commensurability classes.}
\end{enumerate}
\end{example}

\subsection{Standard form for surface automorphisms}
In the remainder of the paper therefore we concentrate on the case
of surfaces $F$ with $\chi(F)<0$. Furthermore, unless we explicitly
say to the contrary, all surfaces $F$ are assumed to be compact and
connected.

A commensurability between automorphisms restricts to a commensurability between the underlying
surfaces. A complete set of commensurability invariants of compact surfaces
are the sign of Euler characteristic, and the property of possessing (or not possessing)
a nonempty boundary.

\begin{lemma}\label{surfaces_commensurable}
Let $F_1$ and $F_2$ be compact surfaces with $\chi<0$. If both or neither have nonempty
boundary, they are commensurable. Otherwise they are incommensurable.
\end{lemma}

The proof is elementary; see e.g.\ \cite{Ma}. Since every compact surface
orbifold with $\chi<0$ is good, the lemma extends to orbifolds.

\begin{notation}
Suppose $\Gamma$ (resp.\ $F'$) is a union of circles  (resp.\ a
compact sub-surface) in $F$. Let $F\setminus \Gamma$ (resp.\
$F\setminus F'$) denote the compact surface obtained by
splitting $F$ along $\Gamma$ (resp.\ removing $\text{int} F'$, the
interior of $F'$).
\end{notation}

Recall the Nielsen-Thurston classification of surface automorphisms.
See e.g.\ \cite{Th1, FLP} for details.

\begin{theorem}[Thurston]\label{NT}
Let $\phi$ be an automorphism of a compact surface $F$. Then the isotopy class of $\phi$ has
a representative (which by abuse of notation we continue to denote by $\phi$) so that either
\begin{enumerate}
\item{$\phi$ has finite order, and $[F,\phi]$ is a Seifert manifold with $\H^2 \times \R$
geometry; or}
\item{$\phi$ is pseudo-Anosov --- i.e.\ $F$ admits a pair of transversely measured singular
foliations $\mathfrak{F}_s$ and $\mathfrak{F}_u$ with measures $\mu_s,\mu_u$, and there is
a real number $\lambda > 1$ called the {\em dilatation} so that $\phi$ takes each
foliation to itself, stretching $\mu_u$ by $\lambda$ and compressing $\mu_s$ by $1/\lambda$ ---
and the interior of $[F,\phi]$ admits a complete hyperbolic structure of finite volume; or}
\item{$\phi$ is reducible --- i.e.\ there is a minimal non-empty embedded $1$-manifold
$\Gamma$ in $F$ with a $\phi$-invariant tubular neighborhood $N(\Gamma)$ such that on each
$\phi$-orbit of $F\setminus N(\Gamma)$ the restriction of $\phi$ is either finite order or
pseudo-Anosov, and $[F,\phi]$ is a $3$-manifold with a JSJ decomposition (whose tori correspond
to the $\phi$ orbits of $\Gamma$) into Seifert fibered and hyperbolic pieces.}
\end{enumerate}
\end{theorem}

In the sequel, we will need more precise control over the normal form of $\phi$ near the
boundary of a subsurface on which $\phi$ is pseudo-Anosov. We say a representative pseudo-Anosov
map $\phi$ on $F$ with boundary is in {\em standard form} if it satisfies the following two
conditions:
\begin{enumerate}
\item{near each boundary circle, two $p$-pronged measured transverse
foliations $(\mathfrak{F}^s, \mu^s)$ and $(\mathfrak{F}^u, \mu^u)$
have the form indicated in Figure~1 (illustrating the case $p$=3); and}
\item{on each $\phi$-orbit on $\partial F$, the restriction of $\phi$ is periodic.}
\end{enumerate}

\begin{center}
\psfrag{a}[]{$\mathfrak{F}_s$} \psfrag{b}[]{$\mathfrak{F}_u$}
\includegraphics[height=2.85in]{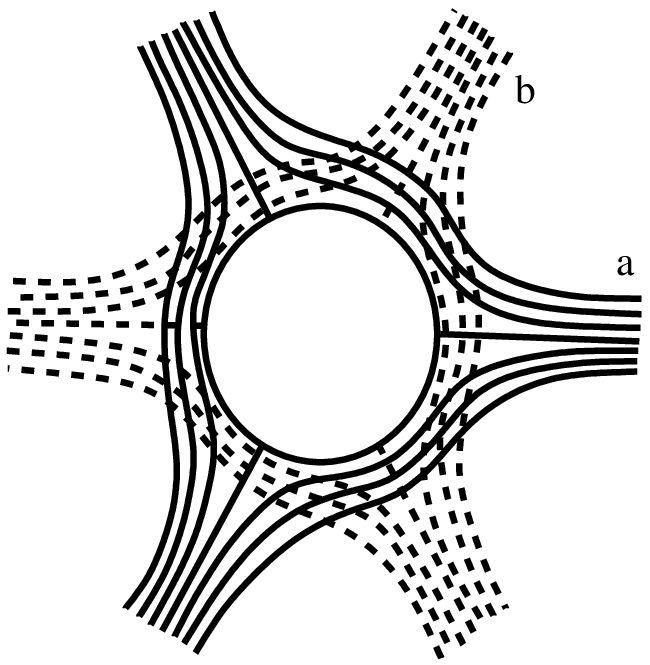}
\centerline{Figure 1}\label{figure_1}
\end{center}

\begin{proposition}[\cite{JG}]\label{JG}
Each reducible map $\phi$ as in case (3) of
Theorem~\ref{NT} can be isotoped into a {\em standard form}; i.e.\:
\begin{enumerate}
\item{the restriction of $\phi$ to each pseudo-Anosov orbit of $F\setminus
N(\Gamma)$ is in standard form as above; and}
\item{the restriction of $\phi$ to each periodic orbit of $F\setminus N(\Gamma)$
is periodic.}
\end{enumerate}
\end{proposition}

This completely fixes the behavior of $\phi$ on the complement of
the regions $N(\Gamma)$. In the sequel we assume that each reducible
map $\phi$ has been isotoped to its standard form in
Proposition~\ref{JG}. Then for any such $\phi$, there is some positive
integer $l$ so that $\phi^l$ is the identity on $\partial
(F\setminus N(\Gamma(\phi)))$ and $\phi$ on $N(\Gamma)$ are Dehn
twists  along each $\gamma\in \Gamma(\phi)$ relative to $\partial
(F\setminus N(\Gamma(\phi)))$.

\begin{definition}
Let $\phi$ be a reducible map. Say $\phi$ is {\em D-type} if it is
generated by Dehn twists along components of $\Gamma(\phi)$; say
$\phi$ is {\em D-type along $\Gamma(\phi)$} if $\phi$ restricts to
the identity along $\partial N(\Gamma(\phi))$.
\end{definition}

\begin{remark} Note that every $\phi$ has a power $\phi^l$ which is
D-type along $\Gamma(\phi)$. Moreover, $\phi$ is a root of D-type,
i.e.\ some power $\phi^l$ is D-type, if and only  if $\phi$ is
periodic on each $\phi$-orbit of $F\setminus N(\Gamma)$.
Alternatively, every  $\phi$ is either a root of D-type or has
pseudo-Anosov $\phi$-orbits.
\end{remark}

Finally we make the following notational convention. We denote surfaces in general
by $F$, $F_i$, $G$ and so on, and use $\Sigma_{g,n}$ to denote the surface of genus $g$
with $n$ boundary components. We sometimes abbreviate $\Sigma_{g,0}$ to
$\Sigma_{g}$.

\subsection{Seifert fibered case}
Finite order automorphisms are very easy to understand.
Suppose $(F_1, \phi_1)$ and $(F_2, \phi_2)$ have finite order, so that
the manifolds $[F_1,\phi_1]$ and $[F_2,\phi_2]$ are Seifert manifolds with a product geometry.
Each $[F_i,\phi_i]$ is finitely covered by a product $F_i \times S^1$.
From Lemma~\ref{surfaces_commensurable} we can deduce:

\begin{proposition}\label{Seifert}\label{seifert_unique_class}
There is exactly one fibered commensurability class of surface bundles of all closed (resp.\ with
torus boundary) Seifert fibered manifolds whose fiber has negative Euler characteristic. This
class contains infinitely many minimal elements.
\end{proposition}
\begin{proof}
All that needs to be proved is that the class contains infinitely
many minimal elements. A key observation is that if
$\tilde{\phi}$ is primitive in $\MCG(\tilde{F})$ and has a fixed
point near which it acts as a rotation through order $p$, the same
is true of any $\phi \in \MCG(F)$ that it covers. This observation
lets us construct infinitely many minimal elements, as follows.

For each genus $g>1$, let $\phi_g$ be a maximum order orientation
preserving periodic map on $\Sigma_g$. Then (see \cite{St}) $\phi_g$
has order $4g+2$ (indeed there is a unique $\Z/(4g+2)\Z$ subgroup of
$\MCG(\Sigma_g)$ up to conjugacy) and has exactly one fixed point,
one periodic orbit of length 2 and one periodic orbit of length
$2g+1$. Clearly $(\Sigma_g, \phi_g)$ is primitive, and $(\Sigma_g,
\phi_g)$ and $(\Sigma_g, \psi)$ cover each other if and only if
$\psi=\phi^q_g$ for $q$ coprime with $4g+2$. Now suppose $(\Sigma_g,
\phi_g)$ covers $(\Sigma_l, \psi)$ with $l\ne g$. Of course, we must
have $l< g$. On the other hand by the observation above, $\psi$ must
have a fixed point near which it acts as a rotation through order
$4g+2$, which implies that $\psi$ is a periodic map on $\Sigma_l$ of
order at least $4g+2$, which is impossible. This completes the
proof.
\end{proof}

\section{Pseudo-Anosov automorphisms}\label{hyperbolic_section}

\subsection{Minimal elements}

The most important fact we prove about commensurability of pseudo-Anosov automorphisms --- equivalently,
of fibered commensurability of hyperbolic fibered pairs --- is the existence of finitely many
minimal elements in each commensurability class. In fact, working in the orbifold category, the
statement is as clean as it could be:

\begin{theorem}\label{minimal_element_orbifold}
Every commensurability class of hyperbolic fibered pairs contains a unique (orbifold) minimal element.
\end{theorem}

\begin{remark}\label{non-unique}
If $M$ is not arithmetic, then the commensurability class of $M$ (in
the usual sense) contains a unique minimal element which is some
orbifold $O$. However, if $M$ is arithmetic, no such unique minimal
element exists, and the commensurator of $\pi_1(M)$ is dense in
$\PSL(2,\C)$ (see \cite{Borel, Margulis}).
\end{remark}

\begin{remark}
Compare with Proposition~\ref{seifert_unique_class} to see that the hypothesis of ``hyperbolic'' is essential here
(in fact, the hyperbolic world is essentially the only context in which there are unique minimal
elements in a commensurability class).
\end{remark}

We now give the proof of Theorem~\ref{minimal_element_orbifold}.
\begin{proof}
Let $(M,\F)$ be a fibered pair, and after passing to a $2$-fold cover if necessary, assume that
$M$ fibers over $S^1$ with fibers the leaves of $\F$. Thus $M$ has the structure of an $F$-bundle
over $S^1$ with monodromy $\phi$, for some compact surface $F$, and some pseudo-Anosov
homeomorphism $\phi:F \to F$. The suspension of the product structure gives a pseudo-Anosov
flow $X$ transverse to $\F$, with finitely many closed singular orbits corresponding to the
singular points of $\phi$. The interior of the manifold $M$ admits a unique complete
singular Sol metric for which the
leaves of $\F$ are Euclidean surfaces with cone singularities on the singular orbits of $X$; see
e.g.\ \cite{Th2} or \cite{FLP} for details.

Pulling back the singular Sol metric on $M$
gives the interior of the universal cover $(\tilde{M},\tilde{\F})$ the structure of a complete
simply-connected singular Sol manifold, for which the leaves of $\tilde{\F}$ are singular
Euclidean planes, and on which $\pi_1(M)$ acts as a discrete finite covolume group of
isometries. Let $\Lambda$ denote the full group of isometries of $\tilde{M}$ with its singular
Sol metric.
\smallskip

{\noindent \bf Claim:} {\em $\Lambda$ is itself a lattice, and it preserves the foliation $\tilde{\F}$}.

\smallskip

We show how the theorem follows from this Claim. Since
$\pi_1(M)\subset \Lambda$ we have the foliation preserving covering
$p: (M, \F)=(\tilde{M},\tilde{\F})/\pi_1(M)\to
(\tilde{M},\tilde{\F})/\Lambda.$ Since $(M, \F)$ is a hyperbolic
surface bundle of finite volume, we conclude that
$(\tilde{M},\tilde{\F})/\Lambda$ is an orbifold fiber pair  $(O,\G)$.
Notice that any covering map of fibered pairs
$(\tilde{M},\tilde{\F}) \to (M,\F)$ is isotopic to an isometric
covering of the interiors in the singular Sol metrics. Then it is
easy to see that for any pair $(M', \F')$ commensurable with  $(M,
\F)$ the group $\pi_1(M')$ embeds into $\Lambda$ in such a way that
$(M', \F')$ covers $(O,\G)$.

\smallskip

Now we prove the Claim. First, it is evident that $\Lambda$
preserves the stratification of $\tilde{M}$ into ``ordinary'' points
(those with a neighborhood isometric to an open set in Sol) and
singular points (those on the lifts of the singular flowlines of
$X$). Moreover, any isometry between open subsets of Sol must
preserve the foliation by Euclidean planes, as can be seen by
appealing to the well-known structure of the point stabilizers in
$\isom(\Sol)$; see e.g.\ \cite[Chap. 3]{Th2}.

Since $\Lambda$ is equal to the group of isometries of the nonsingular part of $\tilde{M}$, it
follows that $\Lambda$ is a Lie group, by the well-known theorem of Myers-Steenrod \cite{MS}.
Hence if $\Lambda$ is not discrete, it must contain a continuous
family of nontrivial isometries. Such isometries can only act on the singular flowlines as
translations. Let $\ell(t)$ and $\ell'(t)$ be two such flowlines, parameterized by length in such
a way that $\ell(t)$ and $\ell'(t)$ are contained in the same singular Euclidean leaf of $\tilde{M}$,
for each $t$. Assume furthermore that for $|t|$ sufficiently small, the points $\ell(t)$ and $\ell'(t)$
can be joined by a unique (nonsingular) Euclidean geodesic in the singular Euclidean leaf containing them.
Then for small $t$, the length of this Euclidean geodesic as a function of $t$ has the form
$\sqrt{e^{2t}x^2 + e^{-2t}y^2}$ for fixed $x$ and $y$; in particular, the length of this Euclidean
geodesic is not locally constant, and therefore (since elements of $\Lambda$ preserve the foliation
by singular Euclidean planes) a continuous family of isometries must fix $\ell$ and $\ell'$ pointwise.
But this implies that $\tilde{M}$ admits no continuous family of nontrivial isometries,
and $\Lambda$ is discrete. Since it contains $\pi_1(M)$, it is therefore a lattice,
as claimed.
\end{proof}

\begin{remark}
If $F$ is closed, $\tilde{M}$ with its singular Sol metric and with its hyperbolic metric are quasi-isometric.
Consequently if $\ell,\ell'$ are two flowlines, the distance function $d(\cdot,\cdot)$ is proper on
$\ell \times \ell'$ and therefore one obtains another proof that $\Lambda$ contains no nontrivial continuous family.
\end{remark}

\begin{remark}
A fibration of $M$ over a circle is determined by an element of
$H^1(M;\Z)$, which is represented by a unique harmonic $1$-form
$\alpha$ in the hyperbolic metric on $M$. A cover
$(\tilde{M},\tilde{\F}) \to (M,\F)$ pulls back the harmonic $1$-form
on $M$ to the corresponding harmonic $1$-form on $\tilde{M}$ (up to
scale), so one can give a slightly different proof of
Theorem~\ref{minimal_element_orbifold} by using the pullback of this
$1$-form to $\H^3$ and arguing that its set of (projective)
symmetries is discrete. Compare with the proof of Theorem~0.1 in
\cite{Agol}.
\end{remark}

The following two corollaries are immediate:

\begin{corollary}\label{volume_bound}
For any positive constant $C$, the set of hyperbolic fibered pairs in a commensurability class whose
underlying $3$-manifold has volume bounded above by $C$ contains only finitely many elements.
\end{corollary}

\begin{proof}
Such a pair corresponds to a finite index subgroup of the orbifold fundamental group of $(O,\G)$
(with notation as in Theorem~\ref{minimal_element_orbifold}) where the index is bounded by
$C/\vol(O)$. Since $\pi_1(O)$ is finitely generated, the number of such subgroups is bounded.
\end{proof}

\begin{corollary}\label{hyperbolic_different_fibrations}
Suppose $M$ is hyperbolic and fibers over $S^1$, and
$\rank(H_1(M))>1$. Then $M$ fibers over $S^1$ in ways representing
infinitely many fibered commensurability classes.
\end{corollary}

\begin{example}
Suppose $(F,\phi)$ is pseudo-Anosov. Let $c$ be an essential simple
closed curve on $F$, and let $\tau_c$ be a Dehn twist along $c$.
Then the automorphisms $(F,\tau_c^l\circ\phi)$ are hyperbolic for
all large $l$, while the volumes of $[F, \tau_c^l\circ\phi]$ are all
bounded by the volume of the cusped manifold $[F,\phi]\setminus
(c\times\{0\})$. By Corollary~\ref{volume_bound},
there are infinitely many commensurability classes among the $(F,
\tau_c^l\circ\phi)$ for large $l$. Of course, it is easy to see
directly in this case that the underlying manifolds fall into
infinitely many commensurability classes (in the usual sense); see
e.g.\ \cite{Anderson}. We give more substantial examples of incommensurable
pseudo-Anosov automorphisms in the next subsection and after.
\end{example}

\begin{remark} One trivial way to produce a hyperbolic $3$-manifold $M$ with many
non-isotopic but commensurable fibrations is just to choose a
$3$-manifold with a large isometry group. We do not know explicit
examples of two commensurable fibrations of a single hyperbolic
$3$-manifold with different genus.
\end{remark}

\subsection{Commensurability invariants}

The following is an incomplete list of elementary commensurability invariants for
pseudo-Anosov automorphisms:

\begin{enumerate}
\item{whether the underlying surface is closed or bounded;}
\item{the commensurability class of the underlying $3$-manifold of $[F,\phi]$.}
\item{the commensurability class of $\log(K)$ where $K$ is the dilatation;}
\item{the set of orders of the singular points of the invariant foliations;}
\end{enumerate}

For later use we say a few words about (3) and (4). First we
make some definitions. For a pseudo-Anosov automorphism
$(F,\phi)$ with a pair of transversely measured singular foliations
$\mathfrak{F}_{s, u}$, we use $\lambda(\phi)>1$ to denote the
dilatation of $\phi$, and $\delta_n(\phi)$ to denote the number of
singularities of degree $n$, then define $\Delta(\phi)$ to be the (infinite)
vector whose coordinates are the $\delta_n(\phi)$.

The first observation to make is that for pseudo-Anosov
automorphisms, $\lambda(*)$ is only affected by dynamical coverings,
and $\Delta(*)$ is only affected by topological coverings.

\begin{lemma}\label{lambda-Delta}
Suppose $(F_1,\phi_1)$, $(F_2,\phi_2)$ are two commensurable pseudo-Anosov maps.
Then  for some $s, s'\in \Q_+$,
\begin{enumerate}
\item{$\log\lambda(\phi_1)= s\log\lambda(\phi_2)$, and
moreover  $\log\lambda(\phi_1)=\log\lambda(\phi_2)$ if
they are topologically commensurable; and}
\item{$\Delta(\phi_1)= s'\Delta(\phi_2),$ and moreover
$\Delta(\phi_1)=\Delta(\phi_2)$ if they are dynamically
commensurable.}
\end{enumerate}
\end{lemma}

\begin{proof}
These facts follow immediately from the definitions (recall
Definition~\ref{topologically_dynamically_definition}; also, (1) 
follows from the proof of Proposition~\ref{Pi-lambda}).
\end{proof}

\begin{example}[Bounded--unbounded]
In \cite{Hironaka}, Remark~4.3, Hironaka gives an example of a pair of automorphisms $\phi_{(1,3)}$
defined on a genus $2$ surface with four boundary components, and $\phi_{(3,4)}$ defined
on a closed genus $3$ surface with the same dilatation. The commensurability
classes of these examples are also distinguished by the orders of the singular points.
\end{example}

\begin{example}\label{3}
Explicit examples of incommensurable fibrations of the same
hyperbolic $3$-manifold are straightforward to construct and
distinguish by Lemma~\ref{lambda-Delta}. For example, in page 4
of \cite{Hironaka}, fibrations of the complement of the link $6_2^2$
in Rolfsen's tables \cite{Rolfsen} are listed, and their singularity
sets do not satisfy the commensurability condition in bullet (2) of
Lemma~\ref{lambda-Delta}.
\end{example}

\begin{example}\label{Different singular orders}
Incommensurable examples may be obtained by branched covers. Start
with an Anosov automorphism $\phi$ of a torus $T$ with dilatation
$K$, and let $P$ be a finite subset of $T$ permuted by $\phi$. Let
$F$ be obtained as a branched cover of $T$, branched over $P$. Then
some power of $\phi$ lifts to an automorphism of $F$ with dilatation
a power of $K$. Different choices of branch orders give rise to
incommensurable automorphisms of closed surfaces with the same
dilatations, but usually incommensurable singular sets.
\end{example}

One may define a more subtle invariant of commensurability as
follows. Let $\phi$ be a pseudo-Anosov automorphism of $F$, with
measured foliations $\mathfrak{F}_{s,u}$ and projectively invariant
transverse measures $\mu_{s,u}$, and singular set $S$ (note that $S$ is finite).
For any pair of points $p$ and $q$
(possibly $p=q$) in the singular set, and any homotopy class of paths
$\gamma$ from $p$ to $q$ in the complement $F\setminus S$ we define a
number $\ell(\gamma)$ to be the infimum, over all paths $\gamma'$
from $p$ to $q$ which are homotopic to $\gamma$ in $F\setminus S$ rel.\
endpoints, of the product
$$\ell(\gamma) = \inf_{\gamma'} \mu_s(\gamma')\mu_u(\gamma')$$
This number depends on the choice of measures $\mu_s,\mu_u$ in their projective class, but
is well-defined if we normalize the product of measures so that $\int_F d\mu_sd\mu_u = -\chi(F)$.

\begin{definition}
Define the {\em spectrum} of $(F,\phi)$ to be the set of numbers $\ell(\gamma)$ as
$\gamma$ varies over nontrivial homotopy classes of paths in $F\setminus S$ as above.
\end{definition}

\begin{proposition}\label{spectrum_bounded_away_from_zero}
With the normalization of the product of measures as above, the spectrum is a commensurability
invariant. Furthermore, it is strictly positive, and discrete as a subset of $\R$ (and is therefore bounded
away from zero).
\end{proposition}
\begin{proof}
By multiplicativity of Euler characteristic, the normalization of
the product of measures is compatible under finite covers. Each
homotopy class of arcs joining singular points on $F$ lifts to an arc
joining singular points in any cover $\tilde{F}$, so the spectrum as
defined is a commensurability invariant.

It remains to show that the spectrum is discrete. By the properties of a pseudo-Anosov, we have
$\ell(\gamma) = \ell(\phi^i(\gamma))$ for any homotopy class $\gamma$ and any integer $i$.
To show that the spectrum is discrete, it suffices to show that there are
only finitely many $\phi$-orbits of homotopy classes $\gamma$ with $\ell(\gamma)\le C$.

Suppose  $K>1$ is the dilatation of $\phi$, and $\gamma'$ is any path between singular
points on $F$. By the definition of $\mathfrak{F}_{s,u}$, we have
$\mu_s(\phi(\gamma')) = K\mu_s(\gamma')$ and $\mu_u(\phi(\gamma')) =
K^{-1}\mu_u(\gamma')$. So under the automorphism $\phi$, the
difference of their logs changes by $2\log K$. It follows that
whatever the difference of logs is initially, after a suitable power
of $\phi$ the absolute value of the difference can be taken to be at most $\log(K)$.
In other words, there is some integer $i$ so that
$$|\log(\mu_s(\phi^i(\gamma'))) - \log(\mu_u(\phi^i(\gamma')))| \le \log(K).$$

If $A$ and $B$ are positive numbers, then a bound on $AB$ and a bound on
$|\log(A) - \log(B)|$ lets us bound both $A$ and $B$. It follows that if $\ell(\gamma) \le C$
then for some $i$, the homotopy class $\phi^i(\gamma)$ is represented by an
arc $\beta=\phi^i(\gamma')$ for which both $\mu_s(\beta)$ and $\mu_u(\beta)$ are bounded, by a constant
depending only on $C$ and $K$. By the discreteness of $S$, there are only finitely many
such relative homotopy classes $\phi^i(\gamma)$, and each of them has a positive $\ell$
length. So $\ell(\gamma)$ takes only finitely many values in $[0,C]$ (all of them positive).
\end{proof}

\begin{remark}
If $\Sigma$ is a Riemann surface, any quadratic holomorphic differential $\alpha$ on $\Sigma$
defines a pair of singular measured foliations, and we can define a spectrum as above for
a pair $(\Sigma,\alpha)$. Multiplying $\alpha$ by a constant also multiplies the spectrum
by a constant, so we can normalize to quadratic differentials with $\int_\Sigma |\alpha|=1$.
The set of such pairs $(\Sigma,\alpha)$ can be identified with the unit cotangent bundle
in moduli space. The spectrum (defined as above) is constant
on orbits of the Teichm\"uller flow (see e.g.\ \cite{Masur_Tabachnikov} for a definition),
and is discrete (by Proposition~\ref{spectrum_bounded_away_from_zero}) for points on
closed orbits of the flow. For general quadratic differentials the spectrum can have
accumulation points, or its closure can contain a perfect set, or it can even be dense.
\end{remark}

This invariant gives rise to a new way to distinguish commensurability classes of automorphisms.

\begin{example}[Different spectrum]
As above, let $\phi$ be an Anosov automorphism of a torus $T$ (with
a flat metric on the torus of total area $1$). The set of periodic
points is dense, so we can choose two periodic points $O, P$.
The stable and unstable foliations of $\phi$ give coordinates
on $T$, at least in a neighborhood of $O$, so that $O=(0,0)$ and
$P=(x,y)$.

In a suitable cover of $T$ branched over $O$ and $P$ we obtain an automorphism with
dilatation a power of $K$ for which the smallest term in the
spectrum is at most $|xy|$ times a constant depending only on the
combinatorics of the cover. By choosing the periodic point $P$ so that $|xy|$ is
sufficiently small, we can ensure that the first term in the spectrum is
as close to $0$ as we desire, while at the same time fixing the
orders of the singular points. By
Proposition~\ref{spectrum_bounded_away_from_zero}, this construction
gives rise to infinitely many commensurability classes with
commensurable log dilatation and the same combinatorial invariants.
\end{example}

\begin{remark}
Example~\ref{Different singular orders} also produces examples of
infinitely many (incommensurable) pseudo-Anosov maps with different
singular orders but the same spectrum. It is not clear if there exists a pair
of pseudo-Anosov maps with incommensurable log dilatations but the same spectrum.
\end{remark}

\section{Reducible automorphisms}\label{reducible_section}

\subsection{Commensurability invariants of reducible automorphisms}
We have assumed that each reducible map is in its standard form as
described in Proposition~\ref{JG}. We also use the notation from
that proposition without comment.

Let $A$ be an oriented annulus $A$. The mapping class group of $A$
rel.\ boundary is isomorphic to $\Z$, generated by a positive Dehn
twist $\tau$ along the core circle. We denote the $n$th power of
such a Dehn twist by $\tau_n$. Figure 2 shows $n=1$ and $-2$ respectively.

\begin{remark}
In Figure~2 and the figures thereafter, the orientation of the surface is indicated by
a ``cup'' shaped arrow, and the numbered circles on the surface indicate the power of a
positive Dehn twist (with respect to the given orientation).
\end{remark}

\begin{center}
\psfrag{a}[]{$1$} \psfrag{b}[]{$-2$}
\includegraphics[width=3in]{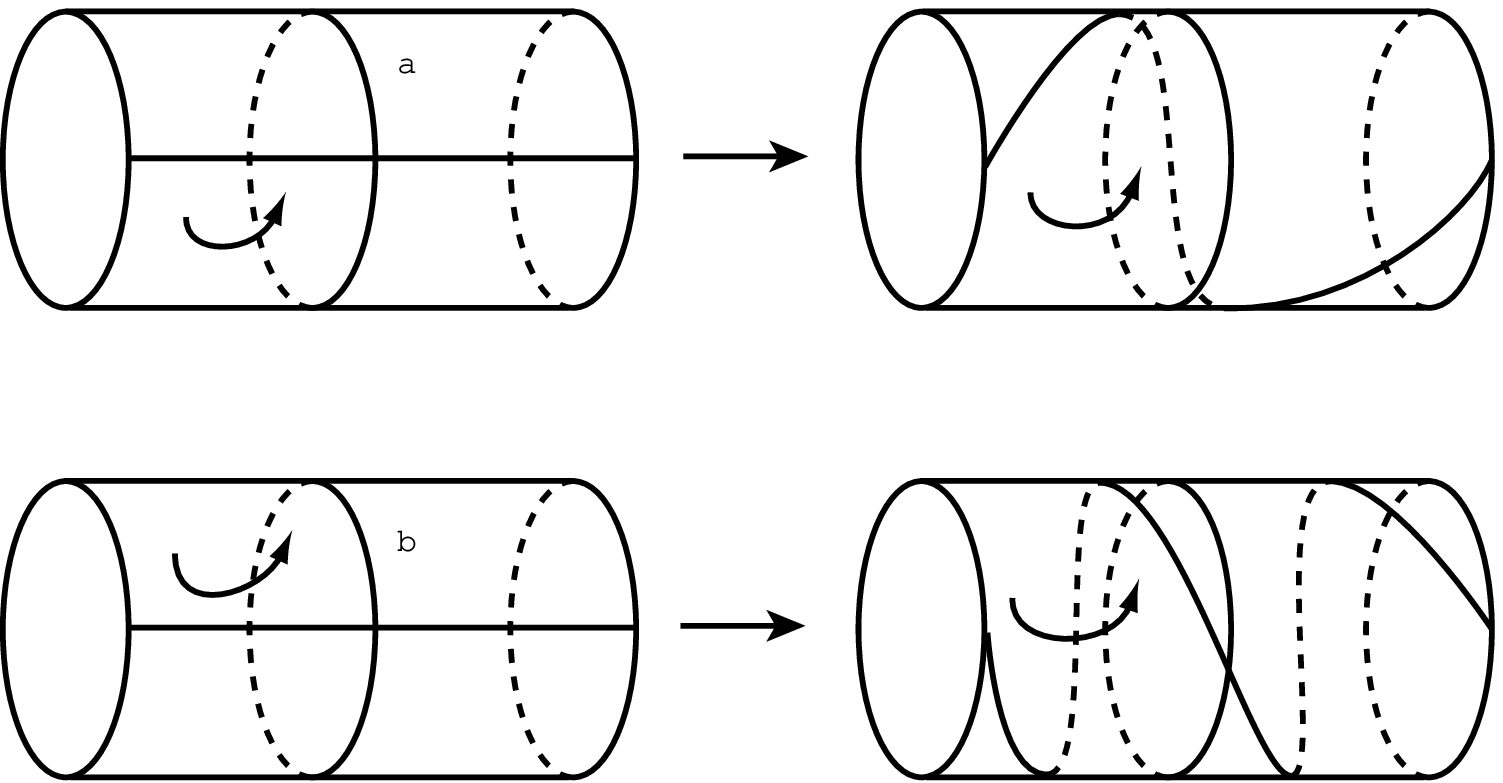}
\vskip 0.5 truecm \centerline{Figure 2}
\end{center}

For a reducible map $\phi$, choose $l$ so that $\phi^l$ is the
identity on $\partial (F\setminus N(\Gamma(\phi)))$. For each
component $N(\gamma)$ of $N( \Gamma(\phi))$, where $\gamma\in
\Gamma(\phi)$, $N(\gamma)$ has the induced orientation and
$\phi^l|\partial N(\gamma)$ is the identity. Then the restriction of
$\phi^l$ to $N(\gamma)$ is the $n$th power of a Dehn twist for some
integer $n$. Now define
$$I(\phi^l,\gamma);\,\,\,I(\phi,\gamma)=I(\phi^l,\gamma)/l;\,\,\,a_k(\phi)=\# \{\gamma\in \Gamma(\phi)|\
I(\phi, \gamma)=k \}, \ k \in \Q$$
Further, define
$$S(\phi)=\{S\ |\ S \,\text { a component of}\, F\setminus
N(\Gamma(\phi))\}$$ and
$$\Omega(S)=\{\gamma\ |\ \gamma \text{ a component of }\partial S\setminus \partial F\}.$$
For every $S \in S(\phi)$,  define
$$a_{S,k}(\phi)= \#
\{\gamma \in\Omega(S) |\ I(\phi, \gamma)=k\}; \,\,\,
A(\phi,S)=(\sum_{k\in \Q_+}
\frac{a_{S,k}(\phi)}{k},\sum_{k\in \Q_-}
\frac{a_{S,k}(\phi)}{-k})$$

The following two numerical invariants are easy to compute:
$$A(\phi)=\frac 12 \sum_{S\in S(\phi)}A(\phi,S)=(\sum_{k\in \mathbb{Q_+}}
\frac{a_k(\phi)}{k},\sum_{k\in \mathbb{Q_-}} \frac{a_k(\phi)}{-k})$$
$$\Pi(\phi)=\{\frac{1}{-\chi(S)}A(\phi,S)\ |\ S \in
S(\phi)\}$$

We say that two sets of ordered pairs of rational numbers $\{(p_i, q_i)\}$ and $\{(p'_j, q'_j)\}$
are equal {\em up to a flip}, denoted $\{(p_i, q_i)\} \flip \{(p'_j, q'_j)\}$,
if either they are equal, or $\{(p_i, q_i)\}=\{(q'_j, p'_j)\}$.
Immediately we have:

\begin{lemma}\label{sign} Reversing the orientation of $F$ preserves
$A(\phi,S)$, and therefore also $A(\phi)$ and $\Pi(\phi)$, up to a flip.
\end{lemma}

We can derive commensurability invariants from $A(\cdot)$ and $\Pi(\cdot)$ as follows:

\begin{theorem}\label{reducible} Suppose $(F_1,\phi_1)$, $(F_2,\phi_2)$ are two reducible maps.
If they are commensurable, then for some $s\in \Q_+$,
$$A(\phi_1)\flip sA(\phi_2)\,\, \text{and}\,\, \Pi(\phi_1)\flip s\Pi(\phi_2).$$
\end{theorem}
We postpone the proof of Theorem~\ref{reducible} until \S~\ref{reducible_subsection}.

\begin{remark}
The invariant $\Pi(\cdot)$ is typically better than $A(\cdot)$ at
distinguishing commensurability classes (though not always; see
Example~\ref{2}).  We say that  a D-type map is {\em definite} if it
is a product of Dehn twists in the components of $\Gamma(\phi)$ of
the same sign.
 Note that the property of having a power which
is definite (along $\Gamma(\phi)$) is a commensurability invariant.
The invariant $A(\cdot)$ can distinguish between definite and
indefinite maps, but can never distinguish different
commensurability classes of definite maps, whereas $\Pi(\cdot)$ can.
\end{remark}

\begin{remark} Both $A(\phi)$ and $\Pi(\phi)$ can be encoded as a polynomial (with
fractional exponents), as follows.
For any pair of non-negative rational numbers $(p,q)$, define
$$S(\phi)(p,q)=\{S \in S(\phi)|
\frac{A(\phi,S)}{-\chi(S)}=(p,q)\};\,\,
\lambda(\phi)_{(p,q)}=\frac{\sum_{S \in
S(\phi)(p,q)}\chi(S)}{\chi(F)}. $$

Now define a polynomial pair:
$$P(\phi)(x,y)=(P_1(\phi)(x,y),P_2(\phi)(x,y))=
\sum_{(p,q)\in\Q^2}(p,q)\lambda(\phi)_{(p,q)}x^py^q$$

One can recover $A(\cdot)$ and $\Pi(\cdot)$ from this polynomial by
the formulae
$$\frac{2}{-\chi(F)}A(\phi)=\sum_{(p,q)\in
\Q^2}(p,q)\lambda(\phi)_{(p,q)}=P(\phi)(1,1),$$
and
$$\Pi(\phi)=\{(p,q)|\ \lambda(\phi)_{(p,q)} \ne 0\}.$$
One can show along lines similar to the proof of Theorem~\ref{reducible} (in the next
subsection) that if two reducible maps $(F_1,\phi_1)$, $(F_2,\phi_2)$ are
commensurable, then for some $s\in \Q_+$, we have
$$P(\phi_1)(x,y)\flip s P(\phi_2)(x^s,y^s)$$
\end{remark}

\subsection{Proof of Theorem~\ref{reducible}}\label{reducible_subsection}
In this subsection we give the proof of Theorem~\ref{reducible}. First we state some
lemmas, that can be verified immediately from the definitions.

\begin{lemma}\label{elementary}
Suppose $\phi$ is a reducible map, then for any positive integer $k$ we have equalities:
\begin{equation}\label{eqn1}
I(\phi^k,\gamma)=kI(\phi,\gamma),\,\,
a_{{S,n}}(\phi^k)=a_{{S,\frac{n}{k}}}(\phi),\,\,
A(\phi^k,S)=\frac{1}{k}A(\phi,S)
\end{equation}
\end{lemma}

\begin{lemma}\label{well2}
Suppose two automorphisms $\phi_1$ and $\phi_2$ on $F$ are isotopic,
and two circles $\gamma_1$ and $\gamma_2$ on $F$ are isotopic. If
$\phi_i$ is D-type along $\gamma_i$, $i=1,2$, then
$I(\phi_1,\gamma_1)=I(\phi_2,\gamma_2)$.
\end{lemma}

From the definitions, from Lemma~\ref{well2} and from the fact that the
reducible system $\Gamma$ is unique up to isotopy (see Theorem~1 in
\cite{Wu} for example), we have

\begin{lemma}\label{isotopy}
$\Pi(\phi)$ and $A(\phi)$ are isotopy  invariants.
\end{lemma}

We now give the proof of Theorem~\ref{reducible}.

\begin{proof}
 Suppose $(F_1, \phi_1)$ and
$(F_2, \phi_2)$ are commensurable. Then there is a surface
$\tilde{F}$, automorphisms $\tilde{\phi}_1$ and $\tilde{\phi}_2$ of
$\tilde{F}$, and nonzero integers $k_1$ and $k_2$, so that
$(\tilde{F},\tilde{\phi}_i)$ covers $(F_i,\phi_i)$ for $i=1,2$, and
$\tilde{\phi}_1^{k_1} = \tilde{\phi}_2^{k_2}$ as automorphisms of
$\tilde{F}$. Denote the covering $ \tilde F \to F_i$ by $p_i$, $
i=1,2$. By Lemma~\ref{sign}, we may assume that the orientations of
$\tilde F$, $F_1$ and $F_2$ have been chosen so that both $p_1,p_2$
are orientation preserving.

Assume first $k_1=k_2=1$ for the moment. By Lemma~\ref{isotopy},
we may assume that $\tilde \phi_1= \tilde \phi_2$ as maps in usual
sense (rather than in their isotopy class).

Consider the following commutative diagram
\[ \begin{CD}
\partial    p_1^{-1}(N(\Gamma(\phi_1)))@> \tilde{\phi_1}^{k}| >>
\partial     p_1^{-1}(N(\Gamma(\phi_1)))\\
@V {p_1}| VV                              @VV p_1|  V\\
\partial  N(\Gamma(\phi_1)) @> \phi_1^{k}| >> \partial N(\Gamma(\phi_1))
\end{CD}  \]
where $k$ is chosen so that  $\phi_1^{k}|_{\partial
N(\Gamma(\phi_1))}=\id|_{\partial N(\Gamma(\phi_1))}$. It follows
that the restriction of $\tilde{\phi_1}^{k}$ to ${\partial p_1^{-1}(N(\Gamma(\phi_1)))}$ is
a deck transformation of the covering $p_1| $. Since $p_1|  $ is
a finite covering, by replacing $k$ by a power if necessary, we can
assume that $\tilde{\phi_1}^{k}$ agrees with $\id$ on $\partial
p_1^{-1}(N(\Gamma(\phi_1)))$ and consequently maps every component of
$p_1^{-1}(N(\Gamma(\phi_1)))$ to itself. For such a $k$, each
$\phi_i^k$, $\tilde{\phi_i}^k$, $i=1,2$ are D-type along their
respective reducible systems, where
$\Gamma(\tilde{\phi_i}^k)=p_i^{-1}(\Gamma(\phi_i))$.

For each $S_1\in S(\phi_1)$ and each component $\tilde S$ of
$p_1^{-1}(S_1)$, there  exists a component $S_2\in S(\phi_2)$, such
that $\tilde S$ is a component of $p_2^{-1}(S_2)$. Assume $p_i|:
\tilde S \to S_i$ are $l_i$ sheeted coverings $i=1,2$.

Pick a component $ \gamma \in \Omega(S_1)$. Suppose
$\{\delta_1,\dots,\delta_t\}=(p_1|\tilde S)^{-1}(\gamma)$, and
$p_1:\ \delta_i \to \gamma$ is a $d_i$ sheeted covering. Then
$\sum_{i=1}^l d_i=l_1$.

Under an $m$-fold covering of annuli, a Dehn twist on the covering annulus projects
to the $m$th power of a Dehn twist on the image annulus. Consequently
$d_iI(\tilde{\phi_1}^k,\delta_i)=I(\phi_1^k,\gamma)$, and by equation~\ref{eqn1} we
have
\begin{equation}\label{eqn2}
I(\tilde{\phi_1}^k,\delta_i)=\frac{kI(\phi_1,\gamma)}{d_i}
\end{equation}
and moreover the $I(\tilde{\phi_1}^k,\delta_i)$ all have the same sign as the
$I(\phi_1,\gamma),\ i=1,\dots,t$ (because $p_1$ preserves
orientation and $k>0$). Suppose $I(\phi_1,\gamma)\ne 0$. Then by equation~\ref{eqn2},
\begin{equation}\label{eqn3}
\sum_{i=1}^t \frac{1}{I(\tilde{\phi_1}^k,\delta_i)}=\sum_{i=1}^t
\frac{d_i}{kI(\phi_1,\gamma)}=\frac{l_1}{kI(\phi_1,\gamma)}
\end{equation}

Now we sum over circles $\delta \in \Omega(\tilde S)$ with positive
$I(\tilde{\phi_1},\delta)$:

\begin{align*}
\sum_{l>0}{\frac{a_{\tilde S,l}(\tilde{\phi_1}^k)}{l}}
&=\sum_{l>0}{\frac{\# \{\delta \in\Omega(S'_1)|I(\tilde \phi^k,
\delta)=l\}}{l}}\\
&=\sum_{\stackrel{\delta \in \Omega(\tilde
S)}{I(\tilde{\phi_1},\delta)>0}}
{\frac{1}{I(\tilde{\phi_1}^k,\delta)}} =\sum_{\stackrel{\gamma_i\in
\Omega(S_1)}{I(\phi_1,\gamma_i)>0}}{\sum_{\delta\in
(p_1|\tilde S)^{-1}(\gamma_i)}{\frac{1}{I(\tilde{\phi_1}^k,\delta)}}}\\
&= \frac{l_1}{k}\sum_{\stackrel{\gamma_i\in
\Omega(S_1)}{I(\phi_1,\gamma_i)>0}}{\frac{1}{I(\phi_1,\gamma_i)}}
=\frac{l_1}{k}\sum_{l>0}{\frac{a_{S_1,l}(\phi_1)}{l}}
\end{align*}
where the penultimate equality follows from equation~\ref{eqn3}.

By a similar computation, $$ \sum_{l<0} {\frac{a_{\tilde
S,l}(\tilde{\phi_1}^k)}{l}}=\frac{l_1}{k}
\sum_{l<0}{\frac{a_{S_1,l}(\phi_1)}{l}}$$
and therefore
\begin{equation}\label{eqn4}
A(\tilde{\phi_i}^k,\tilde S)=\frac{l_i}{k}A(\phi_i,S_i),\ i=1,2
\end{equation}

By equation~\ref{eqn1} we have
\begin{equation}\label{eqn5}
A(\tilde{\phi_1},\tilde S)=kA(\tilde{\phi_1}^k,\tilde S)
=A(\tilde{\phi_2},S_2')
\end{equation}
Since $l_i=\chi(\tilde S)/\chi(S_i)$,  by equations~\ref{eqn4} and \ref{eqn5}, we
get
\begin{equation}\label{eqn6}
\frac{A(\phi_1,S_1)}{-\chi(S_1)}=
\frac{A(\tilde{\phi_1},\tilde S)}{-\chi(\tilde S)}=
\frac{A(\phi_2,S_2)}{-\chi(S_2)}
\end{equation}

From the definition of $\Pi(\cdot)$ we have $\Pi(\phi_2)\subset \Pi(\phi_1)$. By symmetry we
have $\Pi(\phi_2)=\Pi(\phi_1)$.
Summing over all $\Gamma$ in the argument above in place of $\Omega(S_1)$, we get similarly
$$\frac{A(\phi_1)}{\chi(F_1)}=\frac{A(\phi_2)}{\chi(F_2)}$$
From equation~\ref{eqn1} we have
$\Pi(\phi^k)=\Pi(\phi)/k$ and $A(\phi^k)=A(\phi)/k$ and the proof is complete.
\end{proof}
From the proof above immediately we have
\begin{corollary}\label{topologically} If $(F_1,\phi_1)$ and $(F_2,\phi_2)$
are topologically commensurable, then
$$\frac{A(\phi_1)}{\chi(F_1)}\flip \frac{A(\phi_2)}{\chi(F_2)}\,\, \text{and}\,\, \Pi(\phi_1)\flip \Pi(\phi_2).$$
\end{corollary}

\begin{remark} We remind the reader that our invariants are defined for all
reducible maps (and not just D-type examples and their roots). When
reducible maps are not the roots of the D-type maps, then they have
pseudo-Anosov orbits, and we can combine the invariants defined in
\S~3 and in \S~4. For example, see the Proposition below and Example~\ref{5}.
\end{remark}

\begin{proposition}\label{Pi-lambda}
Suppose $(F_1,\phi_1)$, $(F_2,\phi_2)$ are two commensurable
reducible maps. Then for some $s\in \Q_+$,
$$\log\lambda(\phi_1)= s\log\lambda(\phi_2)\,\, \text{and} \,\,\Pi(\phi_1)\flip
s^{-1}\Pi(\phi_2).$$
\end{proposition}

Here we think of $\lambda(\phi)$ for a reducible map $\phi$ as a (possibly empty)
{\em set} of dilatations of the set of restrictions of $\phi$ to its
pseudo-Anosov orbits.

\begin{proof}
From the definition of commensurability, there are positive integers
$k_1$ and $k_2$ such that $(F_1,\phi_1^{k_1})$ and
$(F_2,\phi_2^{k_2})$ are topologically commensurable, both covered
by $(\tilde{F},\tilde{\phi})$. Evidently we have
$\lambda(\phi_1^{k_1})=\lambda(\tilde \phi) =\lambda(\phi_2^{k_2})$,
and therefore $k_1\log\lambda(\phi_1)=\log\lambda(\phi_1^{k_1})
=\log\lambda(\phi_2^{k_2})=k_2\log\lambda(\phi_2)$ and then
$$\log\lambda(\phi_1)=\frac{k_2}{k_1}\log\lambda(\phi_2).$$
On the other hand by Corollary \ref{topologically} and (4.1), we
have $$\frac {\Pi(\phi_1)}{k_1}=\Pi(\phi_1^{k_1})\flip
\Pi(\phi_2^{k_2})=\frac {\Pi(\phi_1)}{k_2}$$
and therefore
$${\Pi(\phi_1)}\flip\frac{k_1} {k_2}{\Pi(\phi_1)}.$$
The Proposition is proved by setting $s=\frac{k_2} {k_1}$.
\end{proof}

\subsection{Examples of reducible automorphisms}\label{example_subsection}

In this section we give several examples, which illuminate the meaning of the invariants
defined above. A D-type map on an oriented $F$ can be indicated pictorially by assigning integers
to disjoint essential simple closed curves on a surface; we use this convention in what follows.

\begin{example}\label{1} 
Dehn twists in separating and non-separating
curves (on the same surface) are commensurable. In Figure~3, Let $\tilde\phi$ be a
D-type automorphism on a surface $F$ of genus 3 generated by full Dehn
twists on circles $c$ and $c'$ as indicated in the figure.

\begin{center}
  \psfrag{1}[]{$\tau_1$}
  \psfrag{2}[]{$\tau_2$}
  \psfrag{a}[]{$c_1$}
  \psfrag{b}[]{$c_2$}
  \psfrag{c}[]{$c$}
  \psfrag{d}[]{$c'$}
  \psfrag{A}[]{$F/\tau_1$}
  \psfrag{B}[]{$F/\tau_2$}
  \psfrag{F}[]{$F$}
  \psfrag{e}[]{$1$}
\includegraphics[width=3.7in]{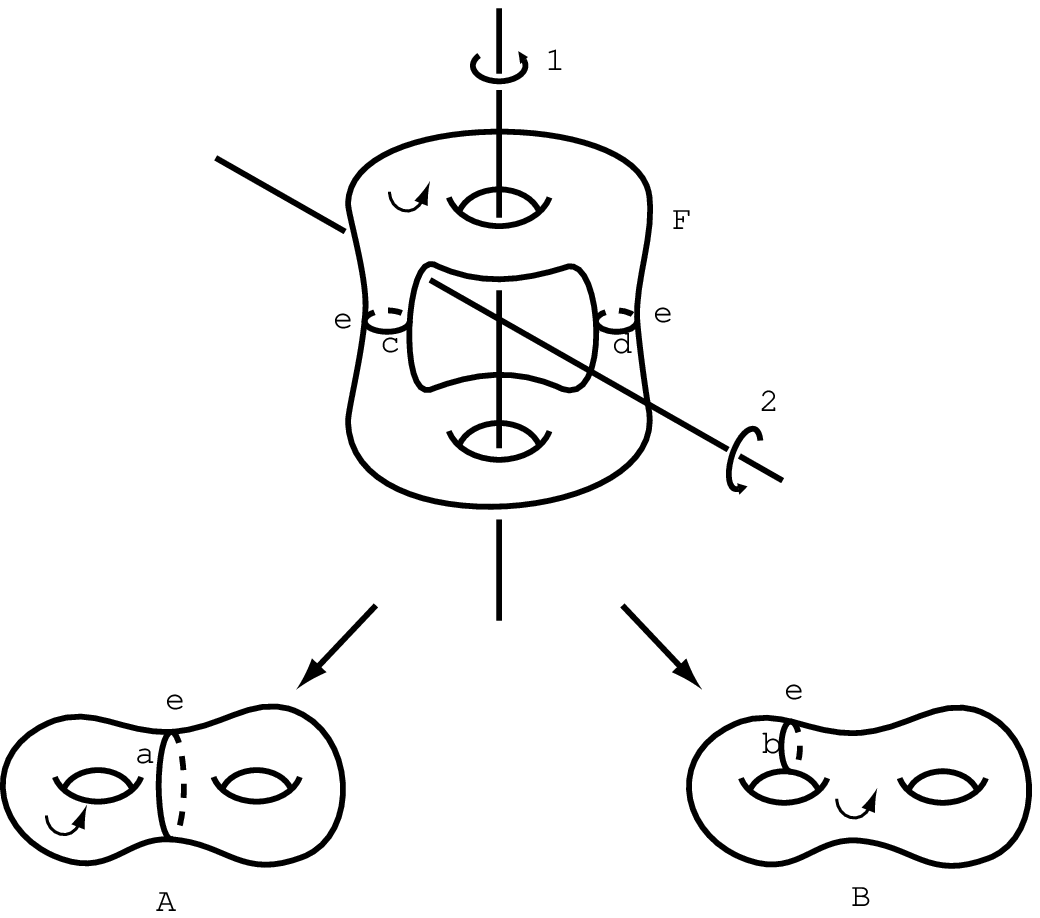}
\vskip 0.5 truecm \centerline{Figure 3}
\end{center}

Then $\tilde\phi$ is invariant under both $\pi$-rotations along $\tau_1$
and $\tau_2$. Hence $\tilde\phi$ induces $\phi_i$ on $F/\tau_i$,
where $\phi_i$ is the Dehn twist along the circle $c_i$. Since $c_1$
is separating while $c_2$ not, $\phi_1$ and $\phi_2$ are not
conjugate. But from the construction they are commensurable.
\end{example}

\begin{example}\label{2} This example show that $\Pi(\phi)$ is not always
finer than $A(\phi)$. Four automorphisms are depicted in Figure~4.
By computing $A(\phi)$ and $\Pi(\phi)$, it can be seen that no pair
of them are commensurable. Notice that on one hand
$A(\phi_1)=A(\phi_2)=({1},{1})$ and
$\{(1,0),(\frac{1}{2},\frac{1}{2}),(0,\frac{1}{3})\}=\Pi(\phi_1)\ne\Pi(\phi_2)=\{(1,0),(\frac{1}{4},\frac{1}{4}),(0,1)\}$,
and on the other hand $(2,1)=A(\phi_3)\ne
A(\phi_4)=({1},\frac{1}{3})$ and
$\Pi(\phi_3)=\Pi(\phi_4)=\{(1,0),(\frac{1}{3},\frac{2}{9})\}$.

\begin{center}
\psfrag{a}[]{$1$} \psfrag{b}[]{$-1$} \psfrag{c}[]{$(F_1,\phi_1)$}
\psfrag{d}[]{$A(\phi_1)=(\frac{1}{6},\frac{1}{6}),\pi(\phi_1)=\{(1,0),(\frac{1}{2},\frac{1}{2}),(0,\frac{1}{3})\}$}
\psfrag{e}[]{$1$} \psfrag{f}[]{$-1$} \psfrag{g}[]{$(F_2,\phi_2)$}
\psfrag{h}[]{$A(\phi_2)=(\frac{1}{6},\frac{1}{6}),\pi(\phi_2)=\{(1,0),(\frac{1}{4},\frac{1}{4}),(0,1)\}$}
\psfrag{k}[]{$1$} \psfrag{l}[]{$1$} \psfrag{m}[]{$-1$}
\psfrag{n}[]{$(F_3,\phi_3)$}
\psfrag{o}[]{$A(\phi_3)=(\frac{1}{5},\frac{1}{10}),\pi(\phi_3)=\{(1,0),(\frac{1}{3},\frac{2}{9})\}$}
\psfrag{p}[]{$1$} \psfrag{q}[]{$-3$} \psfrag{r}[]{$(F_4,\phi_4)$}
\psfrag{s}[]{$A(\phi_4)=(\frac{1}{4},\frac{1}{12}),\pi(\phi_4)=\{(1,0),(\frac{1}{3},\frac{2}{9})\}$}
\psfrag{i}[]{$(a)$} \psfrag{j}[]{$(b)$}
\includegraphics[width=4.99in]{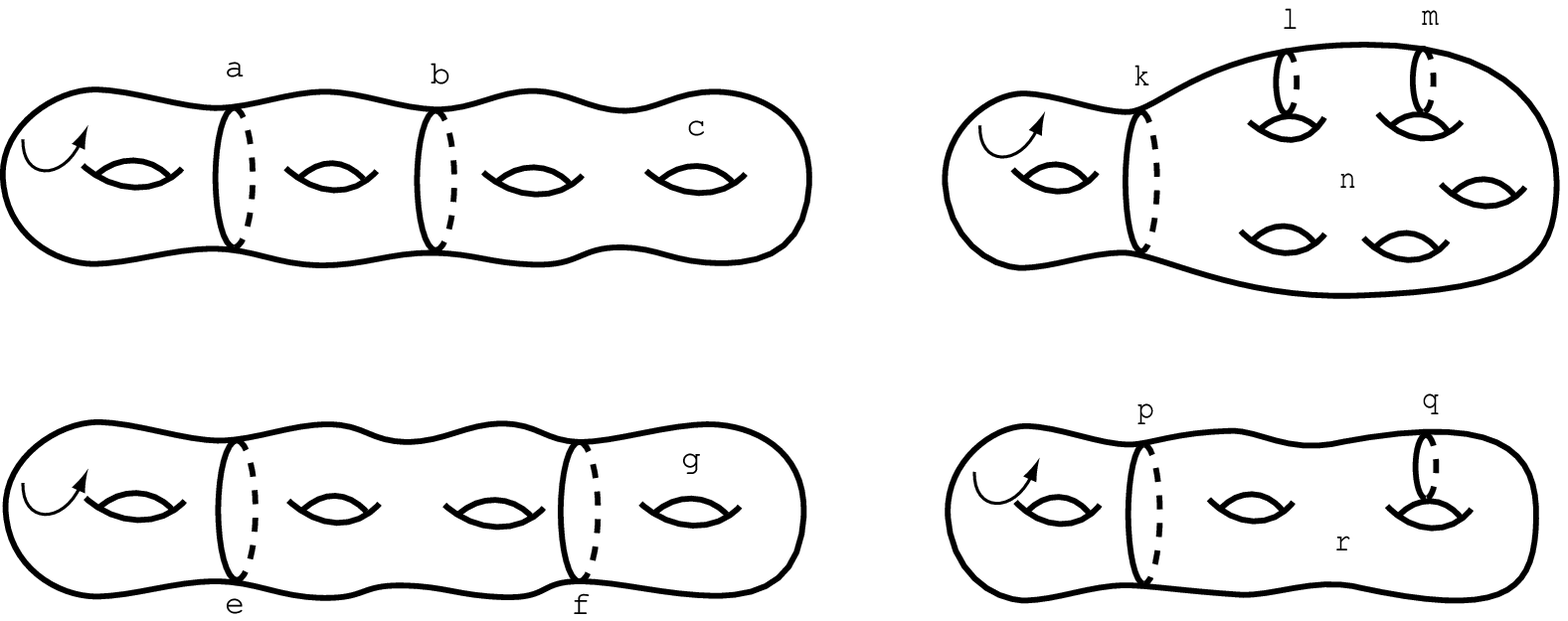}
\vskip 0.5 truecm \centerline{Figure 4}
\end{center}
\end{example}

\begin{example}[Minimal elements]\label{4} Let $\phi_g$ be a
orientation preserving periodic map on $\Sigma_g$ of order $4g+2$
which rotates angle $\frac{\pi}{2g+1}$ around its unique fixed point
$x_g$ (see the proof of Proposition \ref{Seifert}). Remove a
$\phi_g$-invariant disc at $x_g$ from $\Sigma_g$ to get
$\Sigma_{g,1}$.  Connect $\Sigma_{2,1}$ and $\Sigma_{3,1}$ along
their boundaries via an annulus $A$ to form a closed surface
$\Sigma_5$ and define $\phi$ on $F_5$ by
$\phi|\Sigma_{2,1}=\phi_2|\Sigma_{2,1}$ and
$\phi|\Sigma_{3,1}=\phi_3^{-1}|\Sigma_{3,1}$, and then extend to $A$
by a continuous family of rotations through angles from
$\frac{\pi}{5}$ to $\frac{\pi}{7}$. The difference in speeds on the
boundary components is $\frac{2\pi}{35}$, and it follows that
$\phi^{35}$ is a Dehn twist $D_c$. By the uniqueness of the
reducible system and the argument similar in the proof of
Proposition \ref{seifert_unique_class}, one can verify $(\Sigma_5,
\phi)$ is a minimal element. One can construct infinitely many
minimal elements in such a way.
\end{example}

\begin{remark}   One can verify that 35 is
the largest order of a root of a Dehn twists on $\Sigma_5$. It is
amazing that the maximal order of roots of  Dehn twist along {\em
non-separating} curves, which is 11 on $\Sigma_5$ (and in general is
$2g+1$ in $\Sigma_g$), was determined only very recently by several
papers; see \cite{Margalit_Schleimer, McCullough_Rajeevsarathy, Mo}.
\end{remark}

\begin{example}\label{6} This example will be used in \S~\ref{graph_section}.
$\Sigma_{kn+1}$ can be presented as the union of $\Sigma_{1,n}$ and
$n$ copies of $\Sigma_{k,1}$ in a in symmetric way so that there is
an action $\tau_{n,k}$ of order $n$ which acts freely on the triple
$(\Sigma_{kn+1}, \Sigma_{1, n}, \cup_1^n \Sigma_{k, 1})$.

Let $D_c$ be the positive Dehn twist along one component $c$ of
$\partial \Sigma_{1, n}$ and let $\phi_{n,k}$ be the composition of
$D_c\circ \tau_{n,k}$. Then one can verify that $D_{n,k}=\phi_{n,
k}^n$ is D-type, and is given by the product of a positive Dehn
twist along each component of $\partial \Sigma_{1, n}$. For fixed
$k$, the automorphisms $(\Sigma_{kn+1,0},D_{n,k})$ and
$(\Sigma_{km+1}, D_{m,k})$ have a common cover $(\Sigma_{kmn+1},
D_{mn,k})$. Therefore for fixed $k$, $(\Sigma_{kn+1},\phi_{n, k})$
are in the same commensurability class for all $n$.

On the other hand one can verify by inspection that
$\Pi(D_{n,k})=\{(1, 0),(1/(2k-1),0)\}$. So $(\Sigma_{kn+1}, D_{n,k})$
and $(\Sigma_{k'm+1}, D_{m,k'})$ are not commensurable for $k\ne k'$
by Theorem~\ref{reducible}.
\end{example}

\begin{example}\label{||=1}
Each D-type map $(F, \phi)$ is commensurable with a D-type map $(F',
\psi)$ so that the Dehn twist on each $\gamma\in \Gamma(\psi)$ is a
single positive or negative Dehn twist. We can argue as below:

For simplicity, assume $F$ is closed, $ S(\phi)=\{S_i, i=1,...k\}$,
denote $d_\gamma=|I(\phi, \gamma)|$. By replacing $\phi$ by a power
if necessary, we may assume that $d_\gamma$ is an integer $>1$ for
each $\gamma\in \Gamma(\phi)$. Then for each $i$ there is a covering
$q_i: \tilde S_i\to S_i$ such that $q_i|: \tilde \gamma\to \gamma$
is of degree $d_\gamma$ for each component $\gamma\in
\partial S$ and each component $\tilde \gamma$ in
$q_i^{-1}(\gamma)$. One quick way to see this is to attach an
orbifold disk $D_\gamma$ of index $d_\gamma$ to each $\gamma\in
\partial S_i$. The result is $2$-dimensional orbifold which is good, 
since $\chi(S_i)<0$ and each
$d_\gamma>1$. This orbifold has a manifold cover (see
\cite{Thurston_notes}, Chapter~13), and the restriction to $S_i$
gives the required covering $q_i: \tilde S_i\to S_i$.

If $P$ is a planar surface of negative Euler characteristic, then
for every $n\ge 2$ coprime with the number of components of
$\partial P$, there is a cover $\hat{P}\to P$ of degree $n$, which
restricts to a cover of degree $n$ on each boundary component of
$\hat{P}$, and such that $\hat{P}$ is non-planar. Moreover, every
non-planar surface with negative Euler characteristic has a covering
of any given degree which is a covering of degree $1$ on each
boundary component. So after replacing $\phi$ by $\phi^n$, we can
find covers $\hat q_i: \hat S_i \to \tilde S_i$ a covering of degree
$n\prod_{k\ne i} \text{deg}(q_k)$ so that the restriction on each
component of $\partial \hat S_i$ is a covering of degree exactly
$n$. The coverings $p_i=q_i\circ \hat q_i: \hat S_i\to S_i$ match
compatibly to produce a covering $p: \tilde F=\cup \hat S_i \to F$
such that $p|: \tilde \gamma\to \gamma$ is of degree $nd_\gamma$ for
each $\gamma\in \Gamma(\phi)$ and each component $\tilde \gamma$ in
$p^{-1}(\gamma)$. Define a $D$-type map $\tilde \phi$ on $\tilde F$
with $I(\tilde \phi, \tilde \gamma)=1$  if $I(\phi, \gamma)>0$, and
$I(\tilde \phi, \tilde \gamma)= -1$ otherwise, then $\tilde \phi$
covers $\phi$ (see the paragraph before equation~\ref{eqn2} in the
proof of Theorem~\ref{reducible}.
\end{example}

Now we give an application of Proposition \ref{Pi-lambda} to
reducible maps which are not roots of D-type maps.

\begin{example}\label{5}
Let $F$ be a closed oriented surface of genus 2, and $c$ a
non-separating circle in $F$. Let $\phi$ be any pseudo Anosov map on
$F\setminus c$ with dilatation $\lambda(\phi)=K$ and twist angle
$2\pi r$ near $c$, $r\in \Q$, and let $\tau_c$ be a positive Dehn twist
along $c$. Then
\begin{enumerate}
\item{$\tau^{k_1}\circ\phi$ and $\tau^{k_2}\circ\phi$ are
commensurable if and only if $k_1=k_2$; and}
\item{$\tau\circ\phi^{k_1}$ and $\tau\circ\phi^{k_2}$ are
commensurable if and only if $k_1=k_2$.}
\end{enumerate}
The proofs of (1) and (2) are similar; we only give a proof of (1).
Note $\Pi(\tau^{k}\circ\phi)=(1/{(k-r)},0)$,  and
$\lambda(\tau^{k}\circ\phi)=\lambda(\phi)=K>1$ where $r$ and $K$
depends only on $\phi$. If $\tau^{k_1}\circ\phi$ and
$\tau^{k_2}\circ\phi$ are commensurable, by Proposition
\ref{Pi-lambda} and the fact we are considering the automorphism in
the same oriented surface $F$, we should have $\log{K}=
s\log{K}\,\, \text{and} \,\,1/{(k_1-r)}=
s^{-1}/{(k_2-r)}$ for some $s\in \Q_+$. The first equality
implies that $s=1$, and the second implies $k_1=k_2$.
\end{example}

\section{Commensurable and incommensurable bundles in graph manifolds}\label{graph_section}

In this section we give two more complicated examples. The first (Example~\ref{bounded}) is an example
of a graph manifold that is the total space of infinitely many incommensurable fibrations, and at
the same time fibers in infinitely many ways in the {\em same} commensurability class. The
second (Example~\ref{closed}) is an example of a graph manifold that is the total space of infinitely
many incommensurable fibrations, including two incommensurable fibrations with the same genus.
Both examples depend on a construction described in \S~\ref{construction_subsection}.

\subsection{Primary Construction}\label{construction_subsection}

Let $F$ be a compact  oriented surface with the induced orientation
on $\partial F$. Let $a$ be an essential oriented arc on $F$
connecting two different components of $\partial F$. Let $a_0$ and
$a_1$ be the two
 components of the quadrilateral  $\partial N(a)\setminus \partial F$ such that the
direction on $a_0$ induced from the orientation on $\partial N(a)$
is parallel to that on $a$; see Figure 5.

\begin{center}
\psfrag{a}[]{$a_0$} \psfrag{b}[]{$a_1$} \psfrag{c}[]{$a$}
\includegraphics[width=4in]{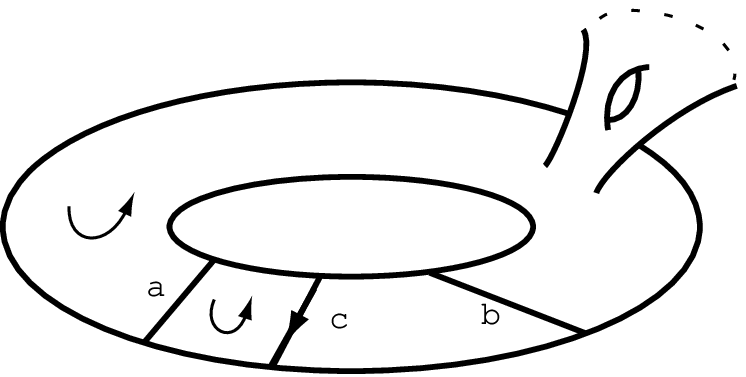}
\vskip 0.5 truecm \centerline{Figure 5}
\end{center}

Then in $F\times [0,1]$, the surface $F\times \{\frac in\}$
intersects the quadrilateral $a_{j}\times [0,1]$ in the arc
$a_{j,i}=a_{j}\times \{\frac in\}$ for each integer $n\ge 2$, where
$j=0,1, i=0,1,\ldots,n.$ Let $A_{1},\ldots,A_{n}$ be  $n$ pairwise
disjoint quadrilaterals properly embedded in $N(a)\times [0,1]$ so
that $A_i$ is a stair connecting $a_{0,i}$ and $a_{1,i+1}$; see
Figure 6(a). 

Let $F_i=(F\times \frac in)\setminus (N(a)\times
[0,1])$ and build a surface $R(a, n)=\cup_{i=0}^{n}F_i \cup
\cup_{l=1}^{n}A_{l}$ in $F\times [0,1]$; see Figure 6(b). A similar
surface $R(\alpha, n)$ in $F\times [0,1]$ can be constructed if we
replace $a$ by a disjoint union of essential arcs $\alpha$ on $F$.

We call the quotient of $R(\alpha, n)$ in $F\times S^1=[F, \text {id}]$
the {\em $n$-floor staircase} along $\alpha$ in $F\times S^1$, or just
{\em $n$-floor along $\alpha$} for short, and denote it as
$F(\alpha, n)$. Note that the surface $F(\alpha, n)$ is transverse to the $S^1$
fibers. If $\alpha$ is empty, then $F(\emptyset, n)$
is just $n$ disjoint copies of $F$ in $F\times S^1$.

Let  $S^1$ have the orientation induced from $[0,1]$. Then  both $F\times S^1$
and $\partial F\times S^1$ are oriented. For each component
$c\in \partial F$, the torus $c\times S^1$ has product coordinates $(c,t)$.
The proof of the following lemma is a routine verification:

\begin{lemma}\label{stair} Let $p: F\times S^1\to F$ be the projection.
Suppose that $\alpha\cap c\le 1$ for each component $c\in \partial F$. Then the
following are true:
\begin{enumerate}
\item{$p: F(\alpha, n)\to F$ is a cyclic covering of degree $n$.
Moreover  $F(\alpha, n)$ is a surface of genus $1-k+n(k-1+g)$ with
$n(\#\partial F-2k)+2k$ boundary components,  where $k=\#\alpha$.}
\item{$p^{-1}(c)$  is either connected or has $n$ components for each
component $c$ of $\partial F$, and $p^{-1}(c)$ is connected if and
only if $\alpha\cap c\ne \emptyset$. Moreover suppose $a$ is an arc
in $\alpha$ with tail in $c'$ and head in $c''$, then $\tilde
 c'=p^{-1}(c')$ has slope $(n, -1)$ and $\tilde c''= p^{-1}(c'')$ has slope
$(n, 1)$.}
\item{Let $\tilde \tau$ be the $2\pi/n$-rotation of $F\times S^1$
along the oriented $S^1$ factor, and let  $\tilde c'$ and $\tilde c''$ be as in
(2). Then $\tau$, the restriction $\tilde \tau$ on $F(\alpha, n)$ is
a generator of the deck group of the covering in (1), which
rotates $\tilde c'$ and $\tilde c''$ through $2\pi/n$ in negative and
positive directions respectively; see Figure 6b.}
\item{$F\times S^1=[F, \id]=[F(\alpha, n),  \tau]$, and
$p_{\alpha, n}: F(\alpha, n)\times S^1=[F(\alpha, n),  \tau^n]\to
F\times S^1=[F, \id]$ is a cyclic covering of degree $n$.}
\end{enumerate}
\end{lemma}

\begin{center}
\psfrag{a}[]{$F\times 1$} \psfrag{b}[]{$F\times \frac{i+1}{n}$}
\psfrag{c}[]{$F\times\frac{i}{n}$} \psfrag{d}[]{$F\times 0$}
\psfrag{e}[]{$a\times 1$} \psfrag{f}[]{$a_{1,i+1}$}
\psfrag{g}[]{$a_{0,i}$} \psfrag{h}[]{$A_i$}
\psfrag{k}[]{$\widetilde{c'}$} \psfrag{l}[]{$\widetilde{c''}$}
\psfrag{m}[]{$c'$} \psfrag{n}[]{$c''$}
\psfrag{p}[]{$\phi$} \psfrag{q}[]{$a$}
\psfrag{i}[]{$(a)$} \psfrag{j}[]{$(b)$}
\includegraphics[width=4in]{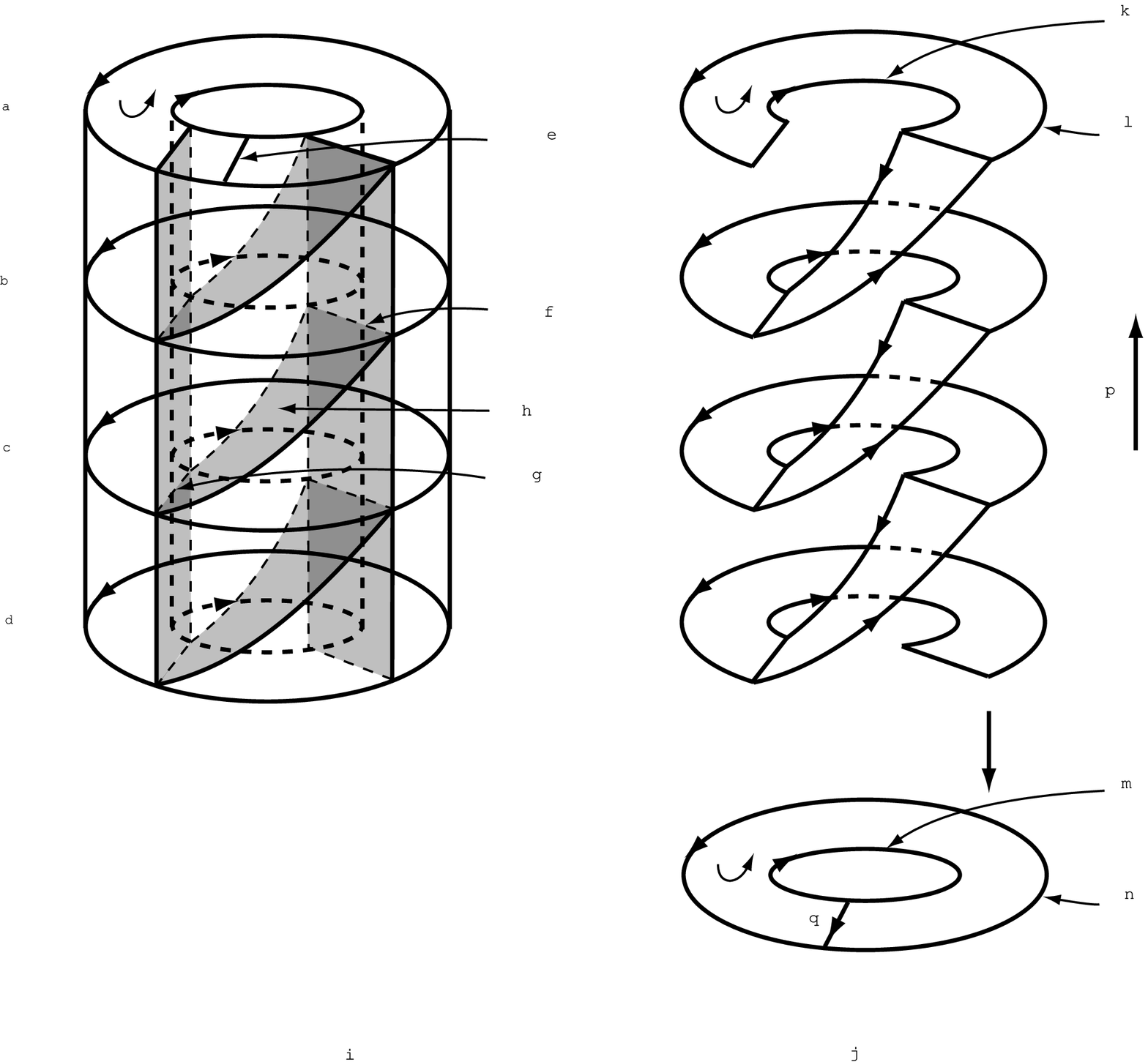}
\vskip 0.5 truecm \centerline{Figure 6}
\end{center}

\begin{remark}\label{circle case}
We can perform a similar construction for a non-separating circle
$\gamma$ in $F$, in which case the description of the
boundary is much simpler: each component of $\partial F$ gives rise to
precisely $n$ copies of $\partial F(\gamma, n)$.
\end{remark}

\subsection{Examples}

\begin{example}\label{bounded}
We describe a graph manifold with the following properties:
\begin{enumerate}
\item{it admits fibrations representing infinitely many fibered commensurability classes;}
\item{it admits infinitely many fibrations representing the {\em same} fibered
commensurability class.}
\end{enumerate}

First take $M=[F_1,\phi_1]$ where the oriented surface $F_1$ and the monodromy
$\phi_1$ are as shown in Figure~7. Note that $M$ has two boundary components
and $\phi$ is D-type and definite.

\begin{center}
\psfrag{a}[]{$S_1$} \psfrag{b}[]{$S_2$} \psfrag{c}[]{$S_3$}
\psfrag{e}[]{$1$} \psfrag{f}[]{$1$}
\includegraphics[width=3.5in]{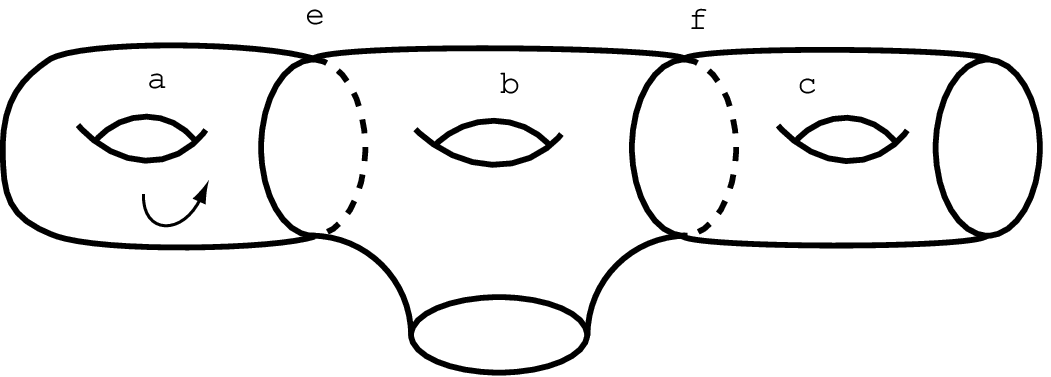}
\vskip 0.5 truecm \centerline{Figure 7}
\end{center}

Another view of $M$ is given in Figure~8, where every
component is of the form $S_i \times S^1$, (depicted in the figure as an $S_i \times I$)
for $i=1,2,3$, and two pairs of boundary tori are identified by maps $f$ and $g$
expressed in terms of coordinates by the maps
$$f(1,0)=(-1,0)\ \ f(0,1)=(-1,1);\ g(1,0)=(-1,0)\ \ g(0,1)=(-1,1)$$
Recall that this notation means that each $(1,0)$ denotes the homotopy class of some
component of some $\partial S_i$, and each $(0,1)$ denotes an $S^{1} \times *$.

\begin{center}
\vskip 0.5 truecm
\psfrag{a}[]{$f$} \psfrag{b}[]{$g$}
\psfrag{e}[]{$S_1\times S^1$} \psfrag{f}[]{$S_2\times S^1$}
\psfrag{g}[]{$S_3\times S^1$}
\includegraphics[width=3.5in]{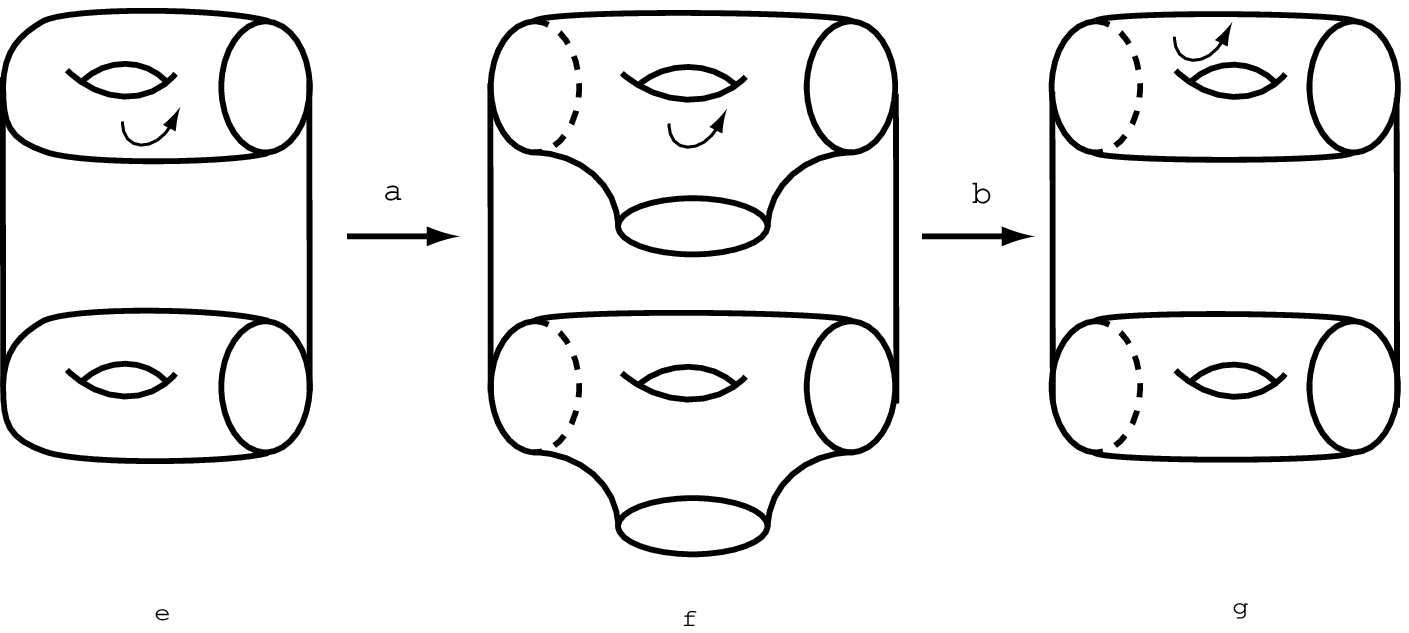}
\vskip 0.5 truecm \centerline{Figure 8}
\end{center}

Now we construct another surface fibration of the same underlying manifold
$M=[F_2,\phi_2]$ as follows. Pick oriented arcs $\alpha_i\in S_i$, $i=2,3$ as in Figure 9.
Then construct $S'_1=S_1(\emptyset, 2)$, $S'_2=S_2(\alpha_2, 2)$,
$S'_3=S_3(\alpha_3, 3)$ in $S_i\times S^1$, $i=1,2,3$, respectively.

\begin{center}
\vskip 0.5 truecm
\psfrag{a}[]{$2$} \psfrag{b}[]{$2$} \psfrag{c}[]{$3$}
\psfrag{e}[]{$S_1(\phi,2)$} \psfrag{f}[]{$S_2(\alpha_2,2)$}
\psfrag{g}[]{$S_3(\alpha_3,3)$}
\psfrag{i}[]{$\alpha_2$} \psfrag{j}[]{$\alpha_3$}
\includegraphics[width=4in]{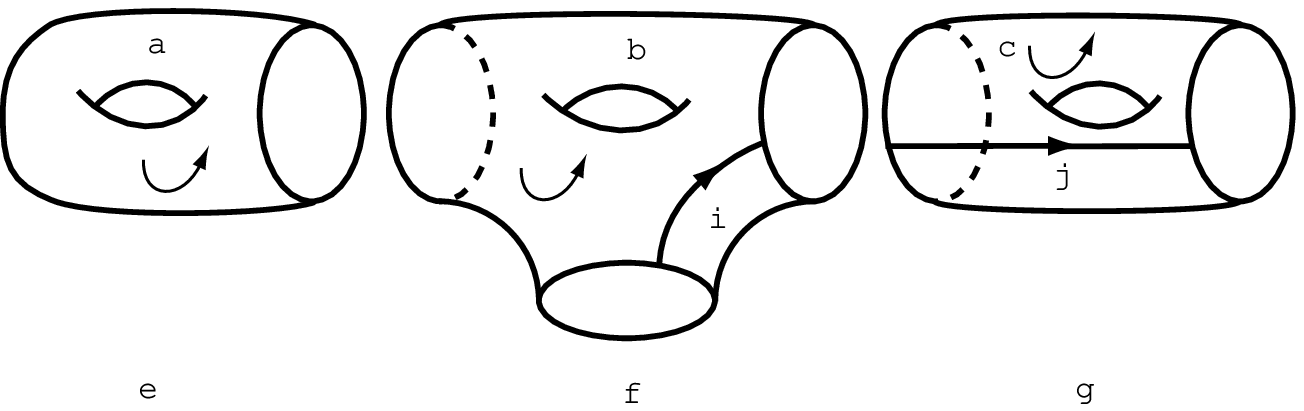}
\vskip 0.5 truecm \centerline{Figure 9}
\end{center}

By Lemma~\ref{stair}~(1), it is easy to see that $S_1'$ is two
copies of $S_1$, that $S_2'$ is a surface of genus 2 with 4 boundary
components, and that $S_3'$ is a surface of genus 3 with 2 boundary
components. By Lemma~\ref{stair}~(2), we see that $\tilde c_2'$ is of slope $(2, 1)$ in
$c_2'\times S^1$, and $\tilde c_3''$ is of slope $(-3, 1)$ in
$c_3''\times S^1$,

Since $g$ sends $(2,1)$ to $(-3,1)$, the maps $f$ and $g$ match $S_1'$,
$S_2'$ and $S_3'$ together to produce a new surface $F_2$ in $M$.
Let $\tau_i$ be the generator of the (cyclic) deck group for the
covering $p_i:S'_i\to S_i$ given by Lemma~\ref{stair}~(3). Then
$\tau_1, \tau_2, \tau_3$ have periods 2, 2, 3 respectively. Now the
new surface bundle structures $[S_i, \tau_i]$ in $S_i\times S^1$
given by Lemma~\ref{stair}~(4), $i=1,2,3$, match to produce a
new surface bundle structure of $M$, which we denote by $[F_2, \phi_2]$.

The monodromy map $\phi_2$ is a virtual D-type automorphism
whose restriction on each $S_i'$ is  $\tau_i$.
Hence $\phi_2$ permutes the two copies of $S_1$ in $F_2$. Moreover
under this permutation, each copy also undergoes a
half-twist relative to $S_2'$. By Lemma~\ref{stair}~(3), $\tau_2$
rotates $\tilde c''_2$ by $\pi$ and $\tau_3$ rotate $\tilde c'_3$ by
$-\frac{2}{3} \pi$ respective along the directions shown in Figure~7.
So the relative twist at $S_2'\cap S_3'$
is $\pi-\frac{2\pi}{3}=\frac{1}{3} \pi$. Now
$\phi_2^6$ is a D-type automorphism as shown in Figure~10.

\begin{center}
\vskip 0.5 truecm
\psfrag{a}[]{$S_1'$} \psfrag{b}[]{$S_2'$} \psfrag{c}[]{$S_3'$}
\psfrag{d}[]{$3$} \psfrag{e}[]{$1$}
\includegraphics[width=3.5in]{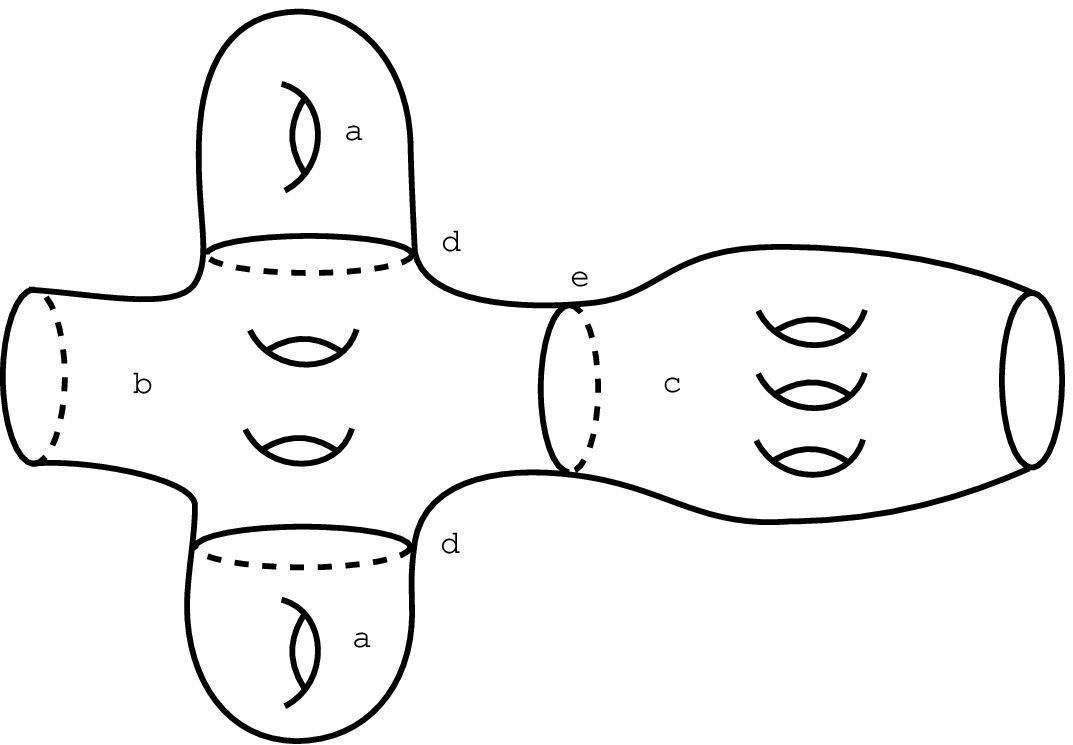}
\vskip 0.5 truecm \centerline{Figure 10}
\end{center}

A direct computation gives
$$\Pi(\phi_1)=\{(1,0),(\tfrac 23, 0), (\tfrac 12, 0) \text{ and }
\Pi(\phi_2)=\{(2,0), (\tfrac{5}{3},0), (1,0)\}$$
Consequently there is no $s \in \Q$
so that $\Pi(\phi_1)\flip s\Pi(\phi_2)$. By Theorem~\ref{reducible}, $(F_1,\phi_1)$ and
$(F_2,\phi_2)$ are not commensurable.

\smallskip
If we perform a similar construction starting from
$S_1(\emptyset, n)$, $S_2(\alpha_2, n)$, $S_3(\alpha_3, n+1)$ in
$S_i\times S^1$, $i=1,2,3$, we will get a surface bundle structure
$[F_n, \phi_n]$ on $M$, where $\phi_n$ is a virtual D-type
automorphism and $\phi_n^{n(n+1)}$ is a D-type automorphism, and
$\Pi(\phi_n)=\{(n, 0), (\frac{2n+1}{3}, 0), (\frac n2, 0) \}$. So
for any positive integers $i\ne j$, the automorphisms
$(F_i,\phi_i)$, $(F_j,\phi_j)$
are not commensurable. We have verified that $M$ fibers in infinitely many
incommensurable ways.

On the other hand if we start from $S_1(\gamma, n)$, $S_2(\emptyset,
n)$ and $S_3(\emptyset, n)$, where $\gamma$ is a non-separating
circle in $S_1$, then by Remark~\ref{circle case} and the argument
above, we can produce a fibration of $M$ with monodromy
$(\Sigma_{2n+1, 2n}, \phi_{2, n})$, where we adapt the notations in
Example~\ref{6}, and use $\Sigma_{2,3} = S_2 \cup S_3$ in place of
$\Sigma_{2,1}$. As observed in Example~\ref{6}, the automorphisms
$(\Sigma_{2n+1, 2n}, \phi_{2, n})$ are commensurable for all $n$. So
$M$ admits infinitely many distinct but commensurable fibrations, as
claimed.
\end{example}

\begin{remark}\label{two boundaries}
One can modify the construction in Example~\ref{bounded} to a more general
setting where the arc connecting two boundary components of $F$ passes through
the cores of more than one Dehn twist. For simplicity, consider a D-type map
which is either a single positive or negative Dehn twist on each $\gamma \in \Gamma(\phi)$
(compare with Example~\ref{||=1}). Then one always gets
infinitely many fibered commensurability classes unless the $\chi(S_i)$
satisfy a certain linear equation so that the invariants in \S~4 fail to distinguish
them, where $S_i$'s are  pieces of $F\setminus \Gamma(\phi)$ meeting the arc.
\end{remark}

\begin{example}\label{closed}
We now give an example of a {\em closed} graph manifold which fibers in infinitely many
incommensurable ways, including two incommensurable fibrations with fibers of the same
genus.

\smallskip

Let  $M=[F, \phi]$ be the graph manifold with $\phi$ as indicated in Figure~11.

\begin{center}
\vskip 0.5 truecm
\psfrag{a}[]{$S_1$} \psfrag{b}[]{$S_2$} \psfrag{c}[]{$S_3$}
\psfrag{d}[]{$2$} \psfrag{e}[]{$-2$} \psfrag{f}[]{$-1$}
\psfrag{g}[]{$1$}
\includegraphics[width=3.5in]{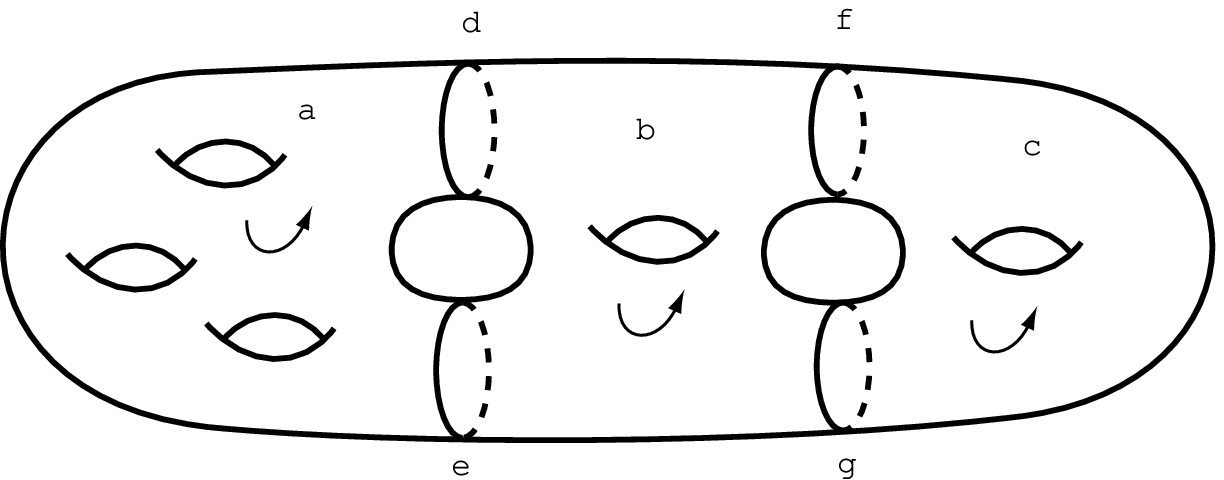}
\vskip 0.5 truecm \centerline{Figure 11}
\end{center}

Our discussion of Figure~8 in Example~\ref{bounded} applies mutatis mutandis
to Figure~12, with gluings given by
\begin{align*}
f_1(1,0)=(-1,0)\ \ f_1(0,1)=(2,1);\ f_2(1,0)=(-1,0)\ \
f_2(0,1)=(-2,1)  \\
g_1(1,0)=(-1,0)\ \ g_1(0,1)=(-1,1);\ g_2(1,0)=(-1,0)\ \
g_2(0,1)=(1,1)
\end{align*}

\begin{center}
\vskip 0.5 truecm
\psfrag{a}[]{$f_1$} \psfrag{b}[]{$f_2$} \psfrag{c}[]{$g_1$}
\psfrag{d}[]{$g_2$}
\psfrag{e}[]{$S_1\times S^1$} \psfrag{f}[]{$S_2\times S^1$}
\psfrag{g}[]{$S_3\times S^1$}
\includegraphics[width=3.5in]{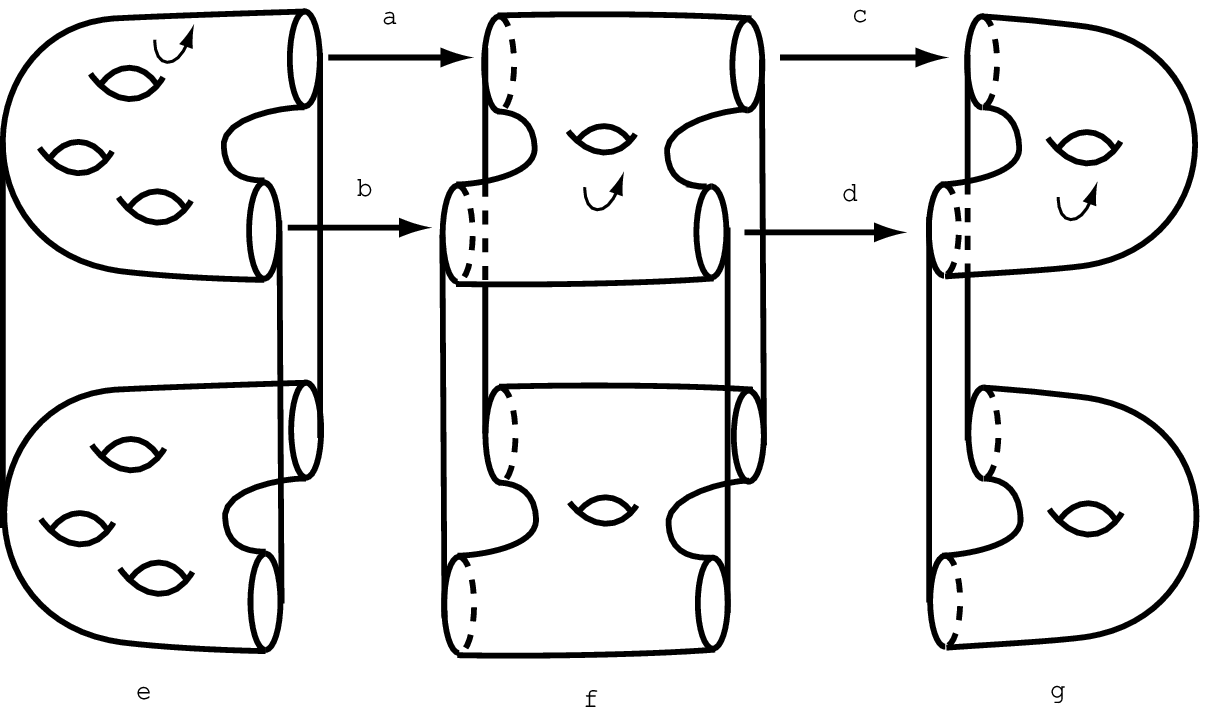}
\vskip 0.5 truecm \centerline{Figure 12}
\end{center}

First we construct infinitely many commensurability classes of fibrations of $M$.

Pick oriented arcs $\alpha_i\in S_i$, $i=1, 2,3$ as in Figure~13 and
 construct $S_1'=S_1(\alpha_1, 4)$, $S'_2=S_2(\alpha_2, 2)$,
$S'_3=S_3(\alpha_3, 3)$ in $S_i\times S^1$, $i=1,2,3$, respectively.
Then $f_i$ and $g_i,\ i=1,2$ paste the boundary of $S_i'$ together
to produce another bundle structure on $M$; i.e.\ we have
$M=[\Sigma_{20}, \phi_2]$, where $\phi_2^{12}$ is a D-type
automorphism on the surface of genus $20$. We can check that
$(\Sigma_{20},\phi_2^{12})$ is as shown in Figure~14 and has
invariant
$\Pi(\phi_2)=\{(\frac{1}{6},\frac{1}{6}),(\frac{5}{4},\frac{5}{4}),(1,1)\}.$

We can perform a similar construction starting from $S_1(\alpha_1,
n+2)$, $S_2(\alpha_2, n)$ and $S_3(\alpha_3, n+1)$ in $S_i\times
S^1$, $i=1,2,3$, and obtain a surface bundle structure
$[\Sigma_{6n+8}, \phi_n]$ on $M$, where $\phi_n^{n(n+1)(n+2)}$ is a
D-type automorphism of a surface of genus $6n+8$ and
$\Pi(\phi_n)=\{(\frac n{12}, \frac n{12}), (\frac{3n+4}{8},
\frac{3n+4}{8}), (\frac n2, \frac n2) \}$. So for any positive
integers $i\ne j$, $(\Sigma_{6i+8},\phi_i)$,
$(\Sigma_{6j+8},\phi_j)$ are incommensurable.

\begin{center}
\psfrag{a}[]{$\alpha_1$} \psfrag{b}[]{$\alpha_2$}
\psfrag{c}[]{$\alpha_3$}  \psfrag{e}[]{$S_1(\alpha_1,4)$}
\psfrag{f}[]{$S_2(\alpha_2,2)$} \psfrag{g}[]{$S_3(\alpha_3,3)$}
\psfrag{i}[]{$4$} \psfrag{j}[]{$2$} \psfrag{k}[]{$3$}
\includegraphics[width=3.5in]{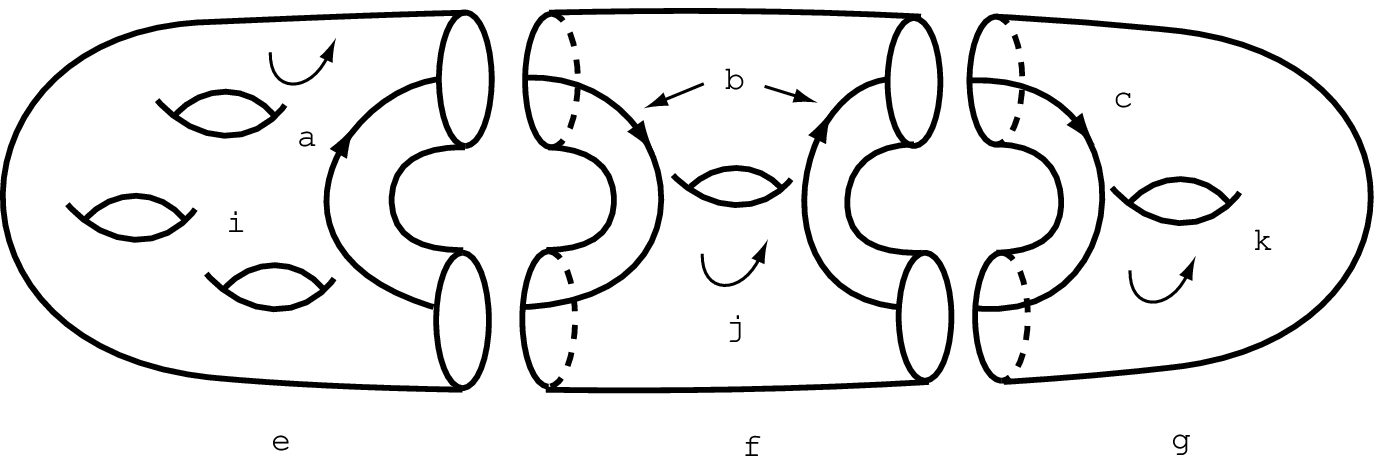}
\vskip 0.5 truecm \centerline{Figure 13}
\end{center}

\begin{center}
\psfrag{a}[]{$S_1'$} \psfrag{b}[]{$S_2'$} \psfrag{c}[]{$S_3'$}
\psfrag{d}[]{$3$}
\psfrag{e}[]{$-3$} \psfrag{f}[]{$-2$} \psfrag{g}[]{$2$}
\psfrag{i}[]{genus 12} \psfrag{j}[]{genus 3} \psfrag{k}[]{genus 3}
\includegraphics[width=3.5in]{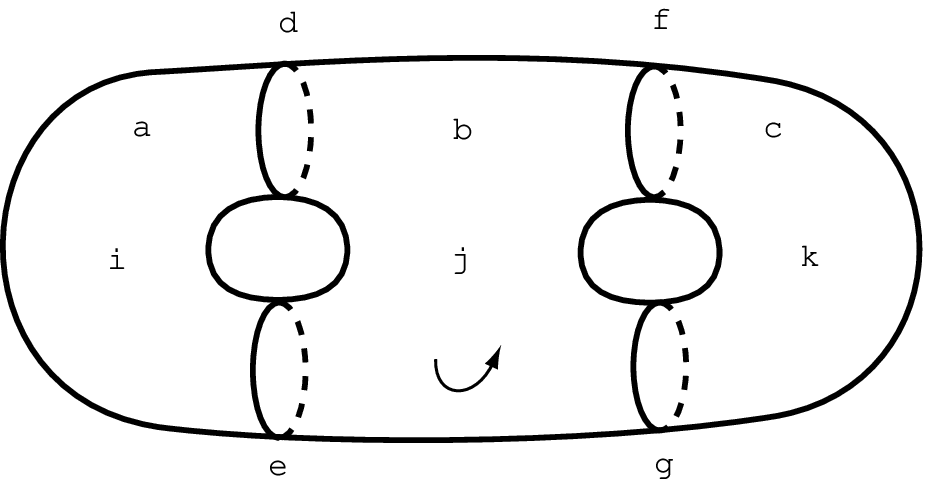}
\vskip 0.5 truecm \centerline{Figure 14}
\end{center}

Now we construct another surface bundle structure $[\Sigma_{20},
\psi]$ on $M$, which is not commensurable with $(\Sigma_{20},
\phi_2)$, where $\phi_2$ is the automorphism above.

Pick oriented arcs $\alpha_i\in S_i$, $i=1, 2,3$ as in Figure~15 and
construct $S_1'=S_1(\emptyset, 3)$, $S'_2=S_2(\alpha_2, 3)$,
$S'_3=S_3(\alpha_3, 4)$ in $S_i\times S^1$, $i=1,2,3$, respectively.
Then $f_i$ and $g_i,\ i=1,2$ glue the boundary of $S_i'$ together to
provide $M$ another structure of surface bundle: $M=[\Sigma_{20},
\psi]$, where $\psi^{12}$ is a D-type automorphism on $\Sigma_{20}$
of genus $20$. We can check that $(\Sigma_{20},\psi^{12})$ is as
shown in Figure~14 and has invariants
$\Pi(\psi)=\{(\frac{1}{4},\frac{1}{4}),(\frac{11}{8},\frac{11}{8}),(\frac{3}{2},\frac{3}{2})\}$.

\begin{center}
\psfrag{a}[]{$\phi$} \psfrag{b}[]{$\alpha_2$}
\psfrag{c}[]{$\alpha_3$} \psfrag{e}[]{$S_1(\phi,3)$}
\psfrag{f}[]{$S_2(\alpha_2,3)$} \psfrag{g}[]{$S_3(\alpha_3,4)$}
\psfrag{i}[]{$3$} \psfrag{j}[]{$3$} \psfrag{k}[]{$4$}
\includegraphics[width=3.5in]{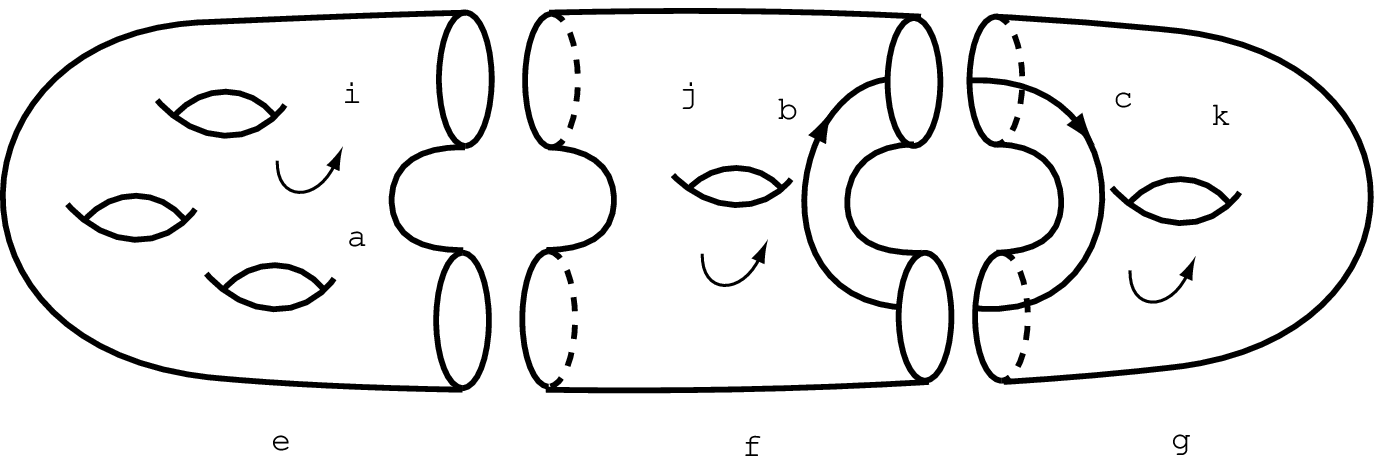}
\vskip 0.5 truecm \centerline{Figure 15}
\end{center}

\begin{center}
\psfrag{a}[]{genus 3} \psfrag{b}[]{genus 4} \psfrag{c}[]{$S_1'$}
\psfrag{d}[]{$S_2'$}
\psfrag{e}[]{$S_3'$} \psfrag{f}[]{$-1$} \psfrag{g}[]{$1$}
\psfrag{i}[]{-8} \psfrag{j}[]{8}
\includegraphics[width=3.5in]{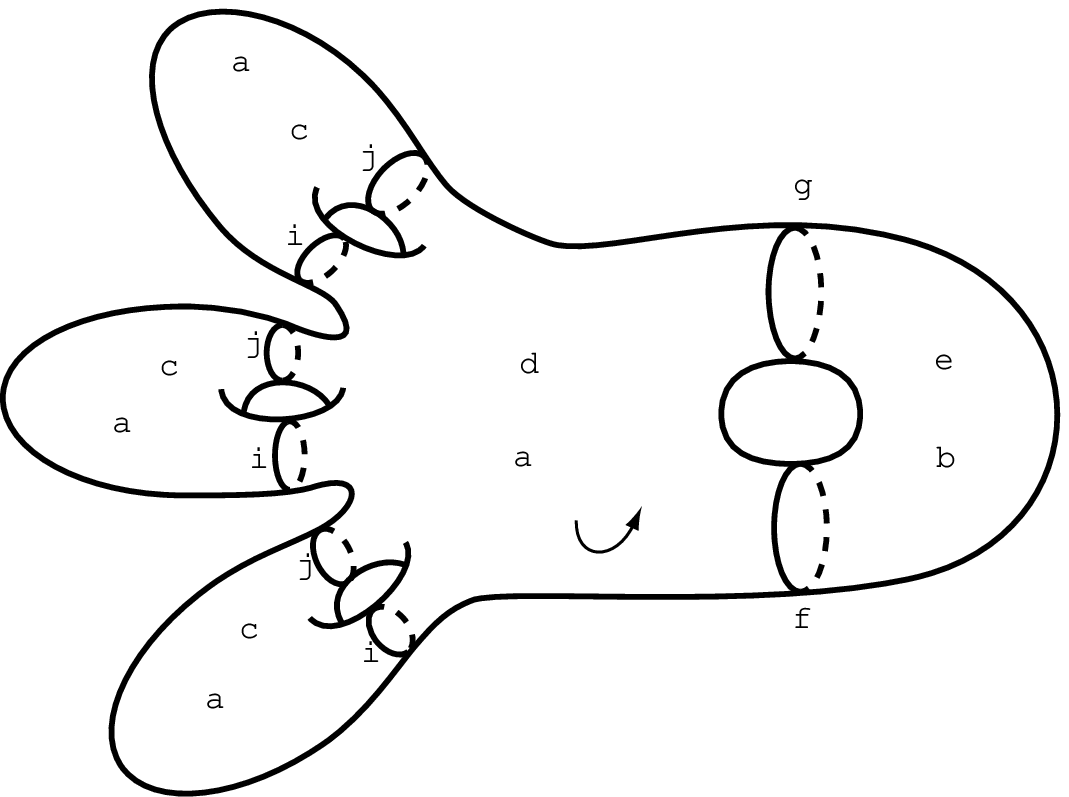}
\vskip 0.5 truecm \centerline{Figure 16}
\end{center}

By Theorem~\ref{reducible} we deduce that $(F_2,\psi)$ and $(F_2,\phi_2)$ are not commensurable,
as claimed.
\end{example}


\begin{thebibliography}{999}
\bibitem{Agol} I.~Agol, {\em Virtual betti numbers of symmetric spaces}, eprint arXiv:math/0611828
\bibitem{Anderson} J.~Anderson, {\em Incommensurability criteria for Kleinian groups},
Proc. AMS {\bf 130} (2002), no. 1, 253--258
\bibitem{Borel} A.~Borel, {\em Commensurability classes and volumes of hyperbolic 3-manifolds.} Ann.
Scuola Norm. Sup. Pisa Cl. Sci. (4) 8 (1981), no. 1, 1--33.
\bibitem{Behrstock_Neumann} J. Behrstock and W. Neumann, {\em Quasi-isometric classification
of non-geometric 3-manifold groups}, eprint arXiv:1001.0212
\bibitem{FLP} A.~Fathi, F.~Laudenbach and V.~Po\' enaru, {\em Travaux de Thurston sur les surfaces}.
Ast\' erisque SMF {\bf 66-67} (1979)
\bibitem{Hironaka} E.~Hironaka, {\em Small dilatation pseudo-Anosov mapping classes
coming from the simplest hyperbolic braid}, eprint arXiv:0909.4517 (corrected version)
\bibitem{JG} B.~J.~Jiang and J.~H.~Guo, Fixed Points of Surface
Diffeomorphisms. Pacific J. Math. 160 (1993), 67-89.
\bibitem{Macbeath} A.~M.~Macbeath, {\em Commensurability of co-compact three-dimensional
hyperbolic groups}, Duke Math. J. {\bf 50} (1983), no. 4, 1245--1253
\bibitem{Margalit_Schleimer} D.~Margalit and S.~Schleimer, {\em Dehn twists have roots},
Geom. Topol. {\bf 13} (2009) 1495-1497.
\bibitem{Margulis} G.~A.~Margulis, {\em Discrete Subgroups of Semisimple Lie
Groups,} Ergebnisse der Mathematik und ihrer Grenzgebiete 17,
Springer-Verlag, New York, 1991.
\bibitem{Ma} W.S.~Massey, {\em Finite covering spaces of $2$-manifolds with boundary.}
Duke Math. J. {\bf 41} (1974), no. 4, 875-887.
\bibitem{McCullough_Rajeevsarathy}
D.~McCullough and K.~Rajeevsarathy, {\em Roots of Dehn twists},
e-print, arXiv:0906.1601v1
\bibitem{Masur_Tabachnikov} H.~Masur and S.~Tabachnikov, {\em Rational billiards and flat structures},
Handbook of dynamical systems, Vol. 1A, 1015--1089, North-Holland, Amsterdam, 2002.
\bibitem{Mo} N.~Monden, {\em On roots of Dehn twists}. e-print arXiv:0911.5079
\bibitem{MS} S.~Myers and N.~Steenrod, {\em The group of isometries of a Riemannian manifold}.
Ann. Math. {\bf 40} (1939), 400--416
\bibitem{Neumann} W.~Neumann, {\em Commensurability and virtual fibration for graph
manifolds}, Topology {\bf 36} (1997), no. 2, 355--378
\bibitem{Nielsen} J.~Nielsen, {\em Jakob Nielsen: collected mathematical papers, Vol. 1},
Contemporary Mathematicians. Birkh\"auser Boston Inc., Boston, MA
1986
\bibitem{Rolfsen} D.~Rolfsen, {\em Knots and links}, Math. Lecture Series No. 7, Publish or Perish,
Berkeley 1976
\bibitem{St} F. ~Steiger,  {\em maximalen Ordnungen periodischer topologischer
Abbildungen geschlossener Flachen in sich}, Comment. Math. Helv. 8
(1935), 48--69.
\bibitem{Th1} W.~Thurston, {\em On the geometry and dynamics of
diffeomorphisms of surfaces}, Bull. Amer. Math. Soc. 19 (2) (1988),
417-438.
\bibitem{Th2}
W.~Thurston, {\em Three-dimensional geometry and topology. Vol.~1}.
Edited by Silvio Levy. Princeton Math. Ser., {\bf 35}, Princeton University Press,
Princeton, NJ, 1997
\bibitem{Thurston_notes} W.~Thurston, {\em The Geometry and Topology of Three-Manifolds},
a.k.a. ``Thurston's notes'', available from the MSRI
\bibitem{Wu} Y.~Q.~Wu,  {\em Canonical reducing curves of surface
homeomorphisms.} Acta Math. Sinica (N.S.) 3 (1987), no. 4, 305--313.

\end{thebibliography}
\end{document}